%% file: deg16.tex
\documentclass{amsart}
 
 \usepackage{amscd}
  \usepackage[curve,matrix,arrow]{xy}
 
 \input skip-thm.tex

\input qformsl-def.tex

 \numberwithin{equation}{section}
 
\numberwithin{figure}{section}
\numberwithin{table}{section}

\newcommand{\ehyp}{e^{\hyp}_3}
\newcommand{\esix}{e^{16}_3}
\newcommand{\ghat}{\hat{g}}
\newcommand{\fe}{\mathfrak{e}}

\newcommand{\sym}{\mathcal{S}}

\newcommand{\cprod}{{\mathinner{\mkern2mu\raise2.5pt\hbox{.}}}}
\newcommand{\cp}{\cprod}

\newcommand{\Qgen}{Q^{\mathrm{gen}}}

\newcommand{\gch}{\check{\gamma}}

\newcommand{\et}{\tilde{\e}}
\newcommand{\phit}{\tilde{\phi}}

\newcommand{\ehat}{{\hat{\eta}}}

\DeclareMathOperator{\ind}{ind}

\renewcommand{\hyp}{{\mathrm{hyp}}}

\newcommand{\rtE}{\mathsf{E_8}}
\newcommand{\rtD}{\mathsf{D_8}}
\newcommand{\rtC}{\mathsf{C_4}}

\DeclareMathOperator{\psp}{PSp}
\newcommand{\PSp}{\psp}
\DeclareMathOperator{\PGL}{PGL}
\DeclareMathOperator{\PSO}{PSO}
\renewcommand{\O}{\mathrm{O}}

\newcommand{\g}{\mathfrak{g}}

\renewcommand{\kill}{\mathrm{Kill}}
\newcommand{\rkill}{\mathrm{redKill}}

\DeclareMathOperator{\hspin}{HSpin}

\newcommand{\IH}{I $\Rightarrow$ H}

\newcommand{\binv}{{\bar{\ }}}

\newcommand{\Cg}{(C, \gamma)}

\begin{document}
 \title[Orthogonal involutions and $E_8$]{Orthogonal involutions on algebras of degree 16 and the Killing form of $E_8$}
 
 \dedicatory{With an appendix by Kirill Zainoulline}
 
 \author{Skip Garibaldi}
\address{(Garibaldi) Department of Mathematics \& Computer Science, Emory University, Atlanta, GA 30322, USA}
\email{skip@member.ams.org}
\urladdr{http://www.mathcs.emory.edu/{\textasciitilde}skip/}

\address{(Zainoulline) Mathematisches Institut der LMU M\"unchen,
Theresienstr.~39,
80333~M\"unchen, Germany}
\email{kirill@mathematik.uni-muenchen.de}

\subjclass[2000]{20G15 (11E88, 17B25)}



\setlength{\unitlength}{.75cm}

\begin{abstract}
We exploit various inclusions of algebraic groups to give a new construction of groups of type $E_8$, determine the Killing forms of the resulting $E_8$'s, and define an invariant of central simple algebras of degree 16 with orthogonal involution ``in $I^3$", equivalently, groups of type $D_8$ with a half-spin representation defined over the base field.  The determination of the Killing form is done by restricting the adjoint representation to various twisted forms of $\PGL_2$ and requires very little computation.

An appendix by Kirill Zainoulline contains a type of ``index reduction" result for groups of type $D$.
\end{abstract}

\maketitle

The first part of this paper (\S\S\ref{In.sec}--\ref{e3.sec}) extends the Arason invariant $e_3$ for quadratic forms in $I^3$ to central simple algebras $\As$ ``in $I^3$" (this term is defined in \S\ref{In.sec}) where $A$ has degree 16 or has a hyperbolic involution.  (The first case corresponds to simple linear algebraic groups of type $D_8$ with a half-spin representation defined over the base field.)  The invariant $e_3$ detects whether $\As$ is generically Pfister, see Cor.~\ref{pfinv} below.  We remark that the paper \cite{BPQ} appears to rule out the existence of such an invariant by a counterexample.  Our invariant exists exactly in the cases where their counterexample does not apply; surprisingly, this includes some interesting cases.  The proofs in this part are not difficult, but we include this material to provide background and context for the later results.  Proposition~\ref{AP} generalizes the Arason-Pfister Hauptsatz for quadratic forms of dimension $< 16$, and depends on a result of Kirill Zainoulline presented in Appendix \ref{kirill.sec}.

The real work begins in the second part of the paper (\S\S\ref{inclusion.sec}--\ref{e3dec.sec}), where we use the inclusion $\PGL_2 \times \psp_8 \subset \PSO_8$ to
give a formula for the Arason invariant  in case $\As$ can be written as a tensor product $(Q, \binv) \ot (C, \gamma)$, where $(Q, \binv)$ is a quaternion algebra with its canonical symplectic involution.

We apply the preceding results in the third part of the paper (\S\S\ref{const.sec}--\ref{conj.sec}) to studying algebraic groups of type $E_8$.  We give a construction of groups of type $E_8$ and compute the Rost invariant, Tits index (in some cases), and Killing form of the resulting $E_8$'s, see Th.~\ref{const.rost}, Prop.~\ref{index}, and Th.~\ref{E8.witt}.  We compute the Killing form by branching to subgroups of type $A_1$, which is somewhat cleaner than computations of other Killing forms in the literature.  

The $E_8$'s arising from our construction are an interesting class.  On the one hand, they are uncomplicated enough to be tractable.  On the other hand, up to odd-degree extensions of the ground field and assuming the validity of a pre-existing conjecture regarding groups of type $D_8$, they include all $E_8$'s whose Rost invariant is zero.  Our construction produces all $E_8$'s over every number field.

\tableofcontents

\subsection*{Notation and conventions} 
We work over a field $k$ of characteristic $\ne 2$. 
Throughout the paper, $\As$ denotes a central simple $k$-algebra with orthogonal involution.  

We often write $\binv$ for the canonical symplectic involution on a quaternion algebra; it will be clear by context which quaternion algebra is intended.  Similarly, we write $\hyp$ for a hyperbolic involution; context again will make it clear whether symplectic or orthogonal is intended.

For $g$ in a group $G$, we write $\Int(g)$ for the automorphism $x \mapsto gxg^{-1}$.

General background on algebras with involution can be found in \cite{KMRT}.  For the Rost invariant, see \cite{MG}.

\part{Extending the Arason invariant to orthogonal involutions}

\section{$I^n$} \label{In.sec}

\begin{defn} \label{In.def}
Let $\As$ be a central simple algebra with orthogonal involution over a field $k$ of characteristic $\ne 2$.  The function field $k_A$ of the Severi-Brauer variety of $A$ splits $A$, hence over $k_A$ the involution $\s$ is adjoint to a quadratic form $q_\s$.  As an abbreviation, we say that $\As$ is \emph{in $I^n k$} (or simply ``in $I^n$") if $q_\s$ belongs to $I^n k_A$, the $n$-th power of the fundamental ideal ideal in the Witt ring of $k_A$.  Clearly, if $\As$ is in $I^n k$, then $(A \ot L, \s \ot \Id)$ is in $I^n L$ for every extension $L/k$.

 We say that $\As$ is \emph{generically Pfister} if $q_\s$ is a Pfister form, or more precisely is \emph{generically $n$-Pfister} if $q_\s$ is an $n$-Pfister form.  
\end{defn}

This first part of the paper is concerned with cohomological invariants of $\As$ in case $\As$ is in $I^3$, especially when $A$ has degree 16.  For context, we give some properties of algebras with involution in $I^3$ or in $I^4$ of small degree.

\begin{eg} \label{In.egs}
\begin{enumerate}
\item $\As$ is in $I$ if and only if $\deg A$ is even.
\item $\As$ is in $I^2$ if and only if it is in $I$ and the discriminant of $\s$ is the identity in $\kx / k^{\times 2}$.
\item \label{In.3} $\As$ is in $I^3$ if and only if it is in $I^2$ and the even Clifford algebra $C_0\As$ is Brauer-equivalent to $A \times k$.
\item Suppose that $\As$ is Witt-equivalent to $(A', \s')$.  Then $\As$ is in $I^n$ if and only if $(A', \s')$ is in $I^n$.
\item \label{pf.deg} Suppose that $\deg A = 2^n$.  Then $\As$ is in $I^n$ if and only if $\As$ is generically Pfister \cite[X.5.6]{Lam}.
\item \label{PFC} Suppose that $\deg A = 2^n$ with $n \ge 2$.  If $\As$ is \emph{completely decomposable} (i.e., isomorphic to a tensor product of quaternion algebras with orthogonal involution), then $\As$ is generically Pfister by \cite{Becher}, hence is in $I^n$.
\end{enumerate}
\end{eg}

Items (1) through (3) show that the property of an algebra being in $I^n$ for $n \le 3$ can be detected by invariants defined over $k$, without going up to the generic splitting field $k_A$.  Below we construct an invariant that detects whether $\As$ belongs to $I^4$ for $A$ of degree 16.

\begin{ques} \label{4pf.ques}
The converse to \eqref{PFC} holds for $n = 1$ (trivial), $n  = 2$ \cite{KPS:16}, and $n  =3$ \cite[42.11]{KMRT}.  Does the converse also hold for $n  = 4$?
That is, \emph{does generically $4$-Pfister imply completely decomposable?}  The answer is ``yes" if $A$ has index 1 (obvious) or 2 \cite[Th.~2]{Becher}.

We return to this question in \S\ref{conj.sec} below.
\end{ques}

\begin{prop}[``Arason-Pfister"] \label{AP}
Suppose that $\As$ is in $I^n$ for some $n \ge 1$ and $\deg A < 2^n$.  If $n \le 4$, then $\s$ is hyperbolic (and $A$ is not a division algebra).
\end{prop}

\begin{proof}
The case where $A$ has index 1 is the Arason-Pfister Hauptsatz.  Otherwise, the Hauptsatz implies that $\s$ is hyperbolic over $k_A$.  If $A$ has index 2 we are done by \cite[Prop.~3.3]{PSS:herm}, and if $\deg A / \ind A$ is odd we are done by Prop.~\ref{kirill}.

The remaining case is where $A$ has degree 8 and index 4.  As $\As$ is in $I^3$, one component of its even Clifford algebra is split and endowed with an orthogonal involution adjoint to an 8-dimensional quadratic form $\phi$ with trivial discriminant.  The involution $\s$ is hyperbolic over $k_A$, hence $\phi$ is also hyperbolic over $k_A$.  As $A$ is Brauer-equivalent to the full Clifford algebra of $\phi$ \cite[\S42]{KMRT}, $\phi$ is isotropic over the base field $k$ by \cite[Th.~4]{Lag:Duke}.  It follows that $\s$ is hyperbolic over $k$ \cite[1.1]{G:clif}.
\end{proof}

The algebras of degree 8 in $I^3$ are completely decomposable by \cite[42.11]{KMRT}.  For degree 10, we have the following nice observation pointed out to us by Jean-Pierre Tignol:
\begin{lem} \label{deg10}
If $\As$ is in $I^3$ and $\deg A \equiv 2 \bmod{4}$, then $A$ is split.
\end{lem} 

In particular, if $\As$ is of degree 10 and in $I^3$, then $A$ is split (by the lemma), hence $\s$ is isotropic by Pfister, see \cite[XII.2.8]{Lam} or \cite[17.8]{G:lens}.

\begin{proof}[Proof of Lemma \ref{deg10}]
Because the degree of $A$ is congruent to 2 mod 4, the Brauer class $\gamma$ of either component of the even Clifford algebra of $\As$ satisfies $2 \gamma = [A]$, see \cite[9.15]{KMRT}.  But $\As$ belongs to $I^3$, so $\gamma$ is 0 or $[A]$.  Hence $[A] = 0$.
\end{proof}

Algebras $\As$ in $I^3$ of degree 12 are described in \cite{GQ12}.  

For $\As$  in $I^3$ of degree 14, the algebra $A$ is split by Lemma \ref{deg10}, hence $\s$ is adjoint to a quadratic form in $I^3$.  These forms have been described by Rost, see \cite{Rost:14.1} or \cite[17.8]{G:lens}.

For $\As$ in $I^3$ and of degree $\ge 16$, the main question to ask is: \emph{How to tell if $\As$ is in $I^4$?}  We address that question in Cor.~\ref{I4.cor} below.

\begin{rmk}[\IH]
In addition to the generically Pfister and completely decomposable algebras with involution, another interesting class of involution are the so-called \emph{\IH} involutions.  We say that a central simple $k$-algebra $A$ with orthogonal involution $\s$ \emph{has \IH} if the degree of $A$ is $2^n$ for some $n \ge 1$, and for every extension $K/k$ over which $\s$ is isotropic, the involution $\s$ is actually $K$-hyperbolic.  If $\As$ has \IH, then $\As$ is generically Pfister, see \cite{BPQ}.  Conversely, if $n \le 4$ and $\As$ is generically Pfister, then $\As$ has \IH\ by the arguments in the proof of Prop.~\ref{AP}.
\end{rmk}

\section{Extending the Arason invariant} \label{arason.sec}

In this section, $\As$ denotes a central simple algebra in $I^3$ over a field $k$.  We use the notation $k_A$ and $q_\s$ from Def.~\ref{In.def}.

In some cases, we can define an element
\begin{equation} \label{e3.1}
e_3\As \in H^3(k, \Zm4) / E(A) 
\end{equation}
for $E(A) := \ker(H^3(k, \Zm4) \ra H^3(k_A, \Zm4))$, such that
\begin{equation} \label{e3.3}
\parbox{4in}{If $K/k$ splits $A$, then $e_3\As$ is the Arason invariant $e_3(q_{\s \ot K})$ in $H^3(k, \Zm4)$.}
\end{equation}
and
\begin{equation} \label{e3.2}
\res_{K/k} e_3 \As = e_3 \left[ \As \ot K \right] \quad \text{for every extension $K/k$.}
\end{equation}

Clearly, properties \eqref{e3.1} and \eqref{e3.3} uniquely determine $e_3\As$ \emph{if such an element exists}.  The existence is a triviality in case $A$ is split.  If $A$ has index 2, then the element $e_3\left[ \As \ot k_A \right] \in H^3(k_A, \Zm4)$ is unramified \cite[Prop.~9]{Berhuy:quat} and so descends to define an element $e_3\As$ as above by \cite[Prop.~A.1]{KRS}.  However, an element $e_3\As$ need not exist if $A$ has degree 8 and is division as Th.~3.9 in \cite{BPQ} shows.  We have:
\begin{thm} \label{e3.thm}
Suppose $\ind A \le 2$ or $2 \ind A$ divides $\deg A$ or $\deg A = 16$.  Then there exists an $e_3\As$ as in \eqref{e3.1} and \eqref{e3.3}.
\end{thm}

To illustrate the cases covered by the proposition, we note that for $A$ of even degree between 8 and 16, the only omitted cases are where $A$ has degree 8 and index 8 or $A$ has degree 12 and index 4.  These cases are genuinely forbidden by \cite{BPQ} and the following example, which extends slightly the reasoning in \cite{BPQ}.

\begin{eg}
Fix a field $k_0$ and an algebra with orthogonal involution $\As$ in $I^3$ over $k_0$, where $A$ has degree 12 and index 4.  By extending scalars to various function fields of quadrics as in \cite{M:simple}, we can construct an extension $k/k_0$ such that $H^3(k, \Zm4)$ is zero and $A \ot k$ still has index 4.  We claim that there is no element $e_3\As$ satisfying \eqref{e3.1} and \eqref{e3.3}.  Indeed, by \eqref{e3.1}, such an $e_3\As$ would be zero.  By \eqref{e3.3}, over the function field of the Severi-Brauer variety of $A$ over $k$, the involution $\s$ becomes hyperbolic.  But this is impossible by Prop.~\ref{kirill}.
\end{eg}

By adding hyperbolic planes, this example and the one from \cite{BPQ} show that there exist $\As \in I^3$ 
\begin{itemize}
\item of index 8 and degree $8 + 16\ell$
\item of index 4 and degree $12 + 8\ell$
\end{itemize}
for all $\ell \ge 0$ such that no element $e_3\As$ satisfies \eqref{e3.1}--\eqref{e3.3}.  Clearly, there are some difficulties for every degree congruent to 4, 8, or 12 mod 16.

\smallskip
As for the proof of Th.~\ref{e3.thm}, the case of index $\le 2$ was treated in \cite{Berhuy:quat}, as outlined above.  In the remaining two cases, we define invariants $\ehyp$ and $\esix$ in \S\ref{e3p.sec} and \S\ref{e3.sec} respectively that take values in $H^3(k, \Zm4) / [A] \cdot H^1(k, \mmu2)$.  Clearly, $[A]\cdot H^1(k, \mmu2)$ is contained in $E(A)$,\footnote{This inclusion is proper for some algebras $A$ of index $\ge 8$ by \cite{Peyre:deg3} and \cite[5.1]{Karp:cod2}.} and 
we define $e_3\As$ to be the image of $\ehyp\As$ or $\esix\As$ in $H^3(k, \Zm4) / E(A)$.  Property \eqref{e3.3} is proved in Examples \ref{e3.eg} and \ref{pfeg} below.  The following corollaries are obvious:

\begin{cor}\label{I4.cor}
An algebra with involution $\As$ as in Th.~\ref{e3.thm} belongs to $I^4$ if and only if $e_3\As$ is zero.$\hfill\qed$
\end{cor}

\begin{cor} \label{pfinv}
An algebra $\As \in I^3$ of degree $16$ is generically Pfister if and only if $e_3\As$ is zero.$\hfill\qed$
\end{cor}

\section{Invariant $\ehyp\As$} \label{e3p.sec}

Suppose that $\As$ is in $I^3$, and 2 times the index of $A$ divides the degree of $A$, i.e., there is a hyperbolic (orthogonal) involution ``$\hyp$" defined on $A$.  We now define an element $\ehyp\As \in H^3(k, \Zm4) / [A] \cdot H^1(k, \mmu2)$ that agrees with the Arason invariant of $\As$ in case $A$ is split.

We may assume that $A$ has degree $\ge 8$; otherwise, $\s$ is hyperbolic by Prop.~\ref{AP} and we set $\ehyp\As = 0$.  We assume further that 4 divides the degree of $A$; otherwise, $A$ has index at most 2 and we set $\ehyp\As$ to be the invariant defined by Berhuy.  These two assumptions imply that $\Spin(A, \hyp)$ is a simple algebraic group of type $D_\ell$ for $\ell$ even and $\ge 4$.

Put $Z$ for the center of $\Spin(A, \hyp)$; it is isomorphic to $\mmu2 \times \mmu2$.  

\begin{lem} \label{hinv.lem}
The sequence
\[
\begin{CD}
H^1(k, Z) @>>> H^1(k, \Spin(A, \hyp)) @>q>> H^1(k, \aut(\Spin(A, \hyp)))
\end{CD}
\]
is exact, and the fibers of $q$ are the $H^1(k, Z)$-orbits in $H^1(k, \Spin(A, \hyp))$.
\end{lem}

If one replaces $\aut(\Spin(A, \hyp))$ with its identity component, then the lemma is obviously true.

\begin{proof}[Sketch of proof of Lemma \ref{hinv.lem}]
Given a 1-cocycle with values in $\Spin(A, \hyp)$, we write $G$ for the group $\Spin(A, \hyp)$ twisted by the 1-cocycle.  The center of $G$ is canonically identified with $Z$, and we want to show that the sequence
\begin{equation} \label{hinv.1}
\begin{CD}
H^1(k, Z) @>>> H^1(k, G) @>q>> H^1(k, \aut(G))
\end{CD}
\end{equation}
is exact.  

Suppose that $\ghat \in H^1(k, G)$ is killed by $q$, i.e., the twisted group $G_{q(\ghat)}$ is isomorphic to $G$.  By the exactness of the sequence
\[
\begin{CD}
1 @>>> \aut(G)^\circ @>>> \aut(G) @>>> \aut(\D) @>>> 1
\end{CD}
\]
for $\D$ the Dynkin diagram of $G$ \cite[\S16.3]{Sp:LAG}, the image $\gamma$ of $\ghat$ in $H^1(k, \aut(G)^\circ)$ is also the image of some $\pi \in \aut(\D)(k)$.  The element $\pi$ acts on $Z$, hence on $H^2(k, Z)$, and since $G_{q(\ghat)}$ is isomorphic to $G$,
\begin{equation} \label{pifix}
\parbox{4in}{The automorphism $\pi$ fixes the Tits class of $G$ in $H^2(k, Z)$.}
\end{equation}

We now show that $\pi$ is in the image of $\aut(G)(k)$; we assume that $\pi \ne 1$.  We write $G$ as $\Spin(A, \tau)$ for some $(A, \tau)$ in  $I^3$.  The even Clifford algebra of $(A, \tau)$ is Brauer-equivalent to $A \times k$.  If $\ell \ne 4$, then \eqref{pifix} implies that $A$ is isomorphic to $M_{2\ell}(k)$ \cite[p.~379]{KMRT}, i.e., $A$ is split, and a hyperplane reflection in $\aut(G)(k)$ maps to $\pi$.  If $\ell = 4$, similar reasoning applies.

Since $\pi$ is in the image of $\aut(G)(k)$, the element $\gamma \in H^1(k, \aut(G)^\circ)$ is zero, hence $\ghat$ comes from $H^1(k, Z)$.  This proves that \eqref{hinv.1} is exact.
\end{proof}

There is a class $\eta' \in H^1(k, \Spin(A, \hyp))$ that maps to the class of $\Spin\As$ in $H^1(k, \aut(G))$, and Lemma \ref{hinv.lem} says that $\eta'$ is determined up to the action of $H^1(k, Z)$.

We put:
\begin{equation} \label{e3.predef}
\ehyp\As := r_{\Spin(A, \hyp)}(\eta') \quad  \in \frac{H^3(k, \Zm4)}{[A] \cp H^1(k, \mmu2)}\ ,
\end{equation}
where $r$ denotes the Rost invariant.
Note that $\ehyp\As$ is well defined by the main result of \cite{MPT}.

\begin{eg} \label{e3.eg}
If $A$ is split, then $\s$ is adjoint to a quadratic form $q_\s$ and $\ehyp\As$ is the Arason invariant of $q_\s$ by \cite[p.~432]{KMRT}.
\end{eg}

\section{$I^3$ and $D_{2n}$}

Write $\Spin_{4n}$ for the split simply connected group of type $D_{2n}$.  Its center is $\mmu2 \times \mmu2$.  Up to isomorphism, $\Spin_{4n}$ has four quotients: itself, $\SO_{4n}$, the adjoint group $\PSO_{4n}$, and one other that we call a half-spin group\footnote{Bourbaki writes ``semi-spin" in \cite{Bou:g7}.}  and denote by $\hspin_{4n}$.  We are interested in it because of the following result:

\begin{lem} \label{class}
The image of $H^1(k, \hspin_{4n})$ in $H^1(k, \PSO_{4n})$ classifies pairs $\As$ of degree $4n$ in $I^3$.
\end{lem}

The algebras of $\As$ in $I^3$ with degree divisible by 4 are in some sense the most interesting ones.  If the degree is not divisible by 4, then $A$ is split by Lemma \ref{deg10}, and classifying such $\As$ amounts to quadratic form theory.

\begin{proof}[Proof of Lemma \ref{class}]
The proof can be summarized by writing: Combine pages 409 and 379 in \cite{KMRT}.

We identify the groups $\Spin_{4n}$ and $\PSO_{4n}$ with the corresponding groups for the split central simple algebra $B$ of degree $4n$ with hyperbolic orthogonal involution $\tau$.  Fix a labeling $C_+ \times C_-$ for the even Clifford algebra $C_0\Bt$.  Write $\pi_+$ for the projection $C_0\Bt \ra C_+$ and $\hspin_{4n}$ for the image of $\Spin_{4n}$ in $C_+$ under $\pi_+$.

Consider the following commutative diagram with exact rows:
\begin{equation} \label{hspin.diag}
\begin{CD}
1 @>>> \mmu2 \times \mmu2 @>>> \Spin_{4n} @>>> \PSO_{4n} @>>> 1 \\
@. @VV{\pi_+}V @VV{\pi_+}V @| @. \\
1 @>>> \mmu2 @>>> \hspin_{4n} @>>> \PSO_{4n} @>>> 1
\end{CD}
\end{equation}
It induces a commutative diagram with exact rows:
\begin{equation} \label{class.1}
\begin{CD}
H^1(k, \Spin_{4n}) @>>> H^1(k, \PSO_{4n}) @>>> H^2(k, \mmu2 \times \mmu2) \\
@VV{\pi_+}V @| @VV{\pi_+}V \\
H^1(k, \hspin_{4n}) @>>> H^1(k, \PSO_{4n}) @>>> H^2(k, \mmu2)
\end{CD}
\end{equation}

As in \cite[p.~409]{KMRT}, the set $H^1(k, \PSO_{4n})$ classifies triples $(A, \s, \phi)$ where $A$ has degree $4n$, the involution $\s$ is orthogonal with trivial discriminant, and $\phi$ is a $k$-algebra isomorphism $Z(C_0\As) \iso Z(C_+ \times C_-)$.  We view $\phi$ as a labeling of the components of the even Clifford algebra of $\As$ as $+$ and $-$.  The image of such a triple $(A, \s, \phi)$ in $H^2(k, \mmu2 \times \mmu2)$ is the Tits class of $\Spin(A, \s, \phi)$, and it follows from \cite[p.~379]{KMRT} and the commutativity of \eqref{class.1} that the image of $(A, \s, \phi)$ in $H^2(k, \mmu2)$ is $C_+\As$.  So $(A, \s, \phi)$ is in the image of $H^1(k, \hspin_{4n})$ if and only if $C_+\As$ is split.  

We have proved that for every $\As$ of degree $4n$ in $I^3$, there is some triple $(A, \s, \phi)$ in the image of $H^1(k, \hspin_{4n}) \ra H^1(k, \PSO_{4n})$.  Suppose now that $(A, \s, \phi)$ and $(A, \s, \phi')$ are in the image of $H^1(k, \hspin_{4n})$; we will show they are equal; we may assume that $\phi \ne \phi'$.  If $A$ is split, then a hyperplane reflection gives an isomorphism between $(A, \s, \phi)$ and $(A, \s, \phi')$, and we are done.  If $A$ is nonsplit, then $C_-\As$ (with numbering given by $\phi$) is nonsplit and it is impossible that $(A, \s, \phi')$ is in the image of $H^1(k, \hspin_{4n})$.  This concludes the proof.
\end{proof}

\section{$\hspin_{16} \subset E_8$} \label{hse}

Write $E_8$ for the split algebraic group of that type.  We view it as generated by homomorphisms $x_\e \!: \Ga \ra E_8$ as $\e$ varies over the root system $\rtE$ of that type, as in \cite{St}.  The root system $\rtD$ is contained in $\rtE$, as can be seen from the completed Dynkin diagrams.  Symbolically, we can see the inclusion as follows.  Fix sets of simple roots $\delta_1, \ldots, \delta_8$ of $\rtD$ and $\e_1, \ldots, \e_8$ of $\rtE$, numbered as in Figure \ref{dynks.fig}, where the unlabeled vertex denotes the negative $-\tilde{\e}$ of the highest root of $\rtE$.
\begin{figure}[ht]
\[
\begin{picture}(7,2)

    \multiput(0.5,1)(1,0){5}{\line(1,0){1}}
    
    \put(0.5,0.45){\makebox(0,0.4)[b]{$\delta_1$}}
    \put(1.5,0.45){\makebox(0,0.4)[b]{$\delta_2$}}
    \put(2.5,0.45){\makebox(0,0.4)[b]{$\delta_3$}}
    \put(3.5,0.45){\makebox(0,0.4)[b]{$\delta_4$}}
    \put(4.5,0.45){\makebox(0,0.4)[b]{$\delta_5$}}
    \put(5.5,0.45){\makebox(0,0.4)[b]{$\delta_6$}}

    \put(5.5,1){\line(1,-1){0.7}}
    \put(5.5,1){\line(1,1){0.7}}

    \multiput(0.5,1)(1,0){6}{\circle*{\darkrad}}

    \put(6.2,0.3){\circle*{\darkrad}}
    \put(6.2,1.7){\circle*{\darkrad}}
    
    \put(6.1,-0.1){\makebox(0,0.4)[b]{$\delta_7$}}
    \put(5.9,1.7){\makebox(0,0.4)[b]{$\delta_8$}}
    
\end{picture}
\begin{picture}(7,3)
    \multiput(1,1)(1,0){8}{\circle*{\darkrad}}
    \put(3,2){\circle*{\darkrad}}

    \put(1,1){\line(1,0){6}}
    \put(3,2){\line(0,-1){1}}
    \multiput(7.3,1)(0.2,0){3}{\circle*{0.01}}

    \put(1,0.45){\makebox(0,0.4)[b]{$\e_1$}}
    \put(2,0.45){\makebox(0,0.4)[b]{$\e_3$}}
    \put(3,0.45){\makebox(0,0.4)[b]{$\e_4$}}
    \put(4,0.45){\makebox(0,0.4)[b]{$\e_5$}}
    \put(5,0.45){\makebox(0,0.4)[b]{$\e_6$}}
    \put(6,0.45){\makebox(0,0.4)[b]{$\e_7$}}
        \put(7,0.45){\makebox(0,0.4)[b]{$\e_8$}}
    \put(3.4,1.85){\makebox(0,0.4)[b]{$\e_2$}}
    \end{picture} 
\]
\caption{Dynkin diagram of $\rtD$ and extended Dynkin diagram of $\rtE$}    \label{dynks.fig}
\end{figure}
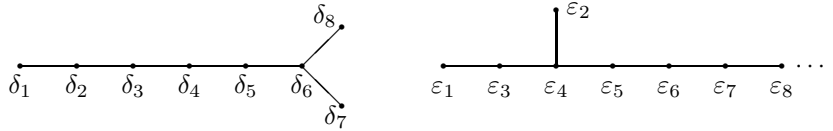
The inclusion of $\rtD$ in $\rtE$ is given by the following table:
\begin{equation} \label{d8e8}
\begin{array}{c|cccccccc}
\text{$\rtD$ root} & \delta_1&\delta_2&\delta_3&\delta_4&\delta_5&\delta_6&\delta_7&\delta_8 \\ \hline
\text{$\rtE$ root} & -\tilde{\e}&\e_8&\e_7&\e_6&\e_5&\e_4&\e_2&\e_3
\end{array}
\end{equation}

The subgroup of $E_8$ generated by the $x_\e$'s for $\e \in \rtD$ is a subgroup of type $D_8$.  This is standard, see e.g.\ \cite{BoSieb}.  Moreover, the subgroup of type $D_8$ is $\hspin_{16}$.  This can be seen by root system computations as in \cite[\S1.7]{Ti:si} or with computations in the centers as in \cite{GQ}.

 It follows from Table \ref{d8e8} (and is asserted in \cite{Dynk:ssub}), that:
\begin{equation} \label{rostmult}
\parbox{4in}{\emph{The composition $\Spin_{16} \ra \hspin_{16} \ra E_8$ has Rost multiplier $1$.}}
\end{equation}

\begin{borel*} \label{ctr.rinv}
Let $\eta \in H^1(k, \hspin_{16})$ map to the class of $\As$ in $H^1(k, \PSO_{16})$, cf.~Lemma \ref{class}.  We write $\hspin\As$ for the group $\hspin_{16}$ twisted by $\eta$; it is the image of $\Spin\As$ in a split component of $C_0\As$.
We find homomorphisms
\[
\Spin\As \ra \hspin\As \ra (E_8)_\eta.
\]
The center of $\hspin\As$ is a copy of $\mmu2$, and we compute the image of an element $c \in \ksq$ under the composition
\[
\ksq = H^1(k, \mmu2) \ra H^1(k, \hspin\As) \ra H^1(k, (E_8)_\eta) \xrightarrow{r_{(E_8)_\eta}} H^3(k, \QZt).
\]
The center of $\Spin\As$ is $\mmu2 \times \mmu2$, and it maps onto the center of $\hspin\As$ via the map $\pi_+$ from the proof of Lemma \ref{class}.  The induced map $H^1(k, \mmu2 \times \mmu2) \ra H^1(k, \mmu2)$ is obviously surjective, so there is some $\gamma \in H^1(k, \mmu2 \times \mmu2)$ such that $\pi_+(\gamma) = (c)$.  By \eqref{rostmult}, we have:
\[
r_{(E_8)_\eta}(c) = r_{\Spin\As} (\gamma),
\]
which is $\pi_+(\gamma) \cdot [A]$ by \cite{MPT}, i.e., $(c) \cdot [A]$.
\end{borel*}

\section{Invariant $\esix\As$ for algebras of degree $16$ in $I^3$} \label{e3.sec}

Let $\As$ be of degree 16 in $I^3$.  Fix a class $\eta \in H^1(k, \hspin_{16})$ that maps to the class of $\As$ in $H^1(k, \PSO_{16})$.  (Here we are using Lemma \ref{class} to know that there is a uniquely determined element in $H^1(k, \PSO_{16})$.)  
Consider the image $r_{E_8}(\eta)$ of $\eta$ under the map
\[
H^1(k, \hspin_{16}) \ra H^1(k, E_8) \xrightarrow{r_{E_8}} H^3(k, \QZt).
\]
Since $\As$ is killed by an extension of $k$ of degree a power of 2, the same is true for $\eta$, hence also for $r_{E_8}(\eta)$.  As $r_{E_8}(\eta)$ is 60-torsion, we conclude that $r_{E_8}(\eta)$ is 4-torsion, i.e., $r_{E_8}(\eta)$ belongs to $H^3(k, \Zm4)$.  We define
\[
\esix\As := r_{E_8}(\eta) \quad \in \frac{H^3(k, \Zm4)}{[A] \cp H^1(k, \mmu2)}
\]

\begin{thm} \label{welldef}
The class $\esix\As$ depends only on $\As$ \emph{(and not on the choice of $\eta$).}
\end{thm}

\begin{proof}
Suppose that $\eta, \eta' \in H^1(k, \hspin_{16})$ map to $\As \in H^1(k, \PSO_{16})$.  We consider the image $\tau(\eta')$ of $\eta'$ in the twisted group $H^1(k, (\hspin_{16})_\eta)$.  Since $\tau(\eta')$ maps to zero in $H^1(k, (\PSO_{16})_\eta)$, it is the image of some $\zeta \in H^1(k, \mmu2)$, where $\mmu2$ denotes the center of $(\hspin_{16})_\eta$.  In the diagram
\begin{equation} \label{twist.diag}
\xymatrix{
H^1(k, \hspin_{16}) \ar[r] \ar[d]_{\cong}^\tau & H^1(k, E_8) \ar[r]^{r_{E_8}}\ar[d]_{\cong}^\tau & H^3(k, \QZt)\ar[d]^{? - r_{E_8}(\eta)} \\
H^1(k, (\hspin_{16})_\eta)  \ar[r] & H^1(k, (E_8)_\eta)\ar[r]^{r_{(E_8)_\eta}}& H^3(k, \QZt),
}
\end{equation}
the left box obviously commutes and the right box commutes by \cite[p.~76, Lemma 7]{Gille:inv}.  Commutativity of the diagram and \ref{ctr.rinv} give that
\[
r_{E_8}(\eta') = r_{E_8}(\eta) + \zeta \cdot [A],
\]
as desired.
\end{proof}

\begin{eg} \label{pfeg}
If $A$ is split, then $\s$ is adjoint to a quadratic form $q_\s$, and $\esix\As$ equals the Arason invariant of $q_\s$.  Indeed, if $A$ is split, then there is a class $\gamma \in H^1(k, \Spin_{16})$ that maps to $\eta$.  Statement \eqref{rostmult} gives:
\[
r_{E_8}(\eta) = r_{\Spin_{16}}(\gamma) = e_3(q_\s).
\]
\end{eg}

\part{The invariant $\esix$ on decomposable involutions}

The purpose of this part is to compute $e_3\As$ in case $\As$ can be written as $(Q, \binv) \ot \Cg$ where $(Q, \binv)$ is a quaternion algebra endowed with its canonical involution and $\Cg$ is a central simple algebra of degree 8 with symplectic involution.  We do this by computing the value of the Rost invariant of $E_8$ on a subgroup $\PGL_2 \times \psp_8 \times \mmu2$; this finer computation will be used in \S\ref{const.sec}.

\section{An inclusion $\PGL_2 \times \psp_8 \times \mmu2 \subset \hspin_{16}$} \label{inclusion.sec}

\begin{borel}{Inclusions} \label{inclusions}
We now describe concretely a homomorphism of groups $\PGL_2 \times \psp_8 \ra \hspin_{16}$.  
Write $S_n$ for the $n$-by-$n$ matrix whose only nonzero entries are 1s on the ``second diagonal", i.e., in the $(j, n+1-j)$-entries for various $j$.  We identify $\Sp_{2n}$ with the symplectic group of $M_{2n}(k)$ endowed with the involution $\gamma_{2n}$ defined by 
\[
\gamma_{2n}(x) = \Int \stbtmat{0}{S_n}{-S_n}{0}^{-1} x^t.
\]
We identify $\Spin_{2n}$ with the spin group of $M_{2n}(k)$ endowed with the involution $\s_{2n}$ defined by
\[
\s_{2n}(x) = \Int (S_{2n}) \, x^t.
\]
These are the realizations of the groups (stated on the level of Lie algebras) given in \cite[\S{VIII.13}]{Bou:g7}.

We now define isomorphisms
\begin{equation} \label{two.iso}
(M_2(k), \gamma_2) \ot (M_8(k), \gamma_8) \iso (M_{16}(k), \s'_{16}) \iso (M_{16}(k), \s_{16}),
\end{equation}
where 
\[
\s'_{16}(x) = \Int \left( \begin{smallmatrix}
0 & 0 & 0 & S_4 \\
0 & 0 & -S_4 & 0 \\
0 & -S_4 & 0 & 0 \\
S_4 & 0 & 0 & 0
\end{smallmatrix} \right)
x^t.
\]
We take the first isomorphism to be the usual Kronecker product defined on the standard basis elements by
\[
E_{ij} \ot E_{qr} \mapsto E_{8(i-1) + q, 8(j-1)+r}.
\]
The second isomorphism is conjugation by the matrix
\[
\left( \begin{smallmatrix}
1_4 & 0 & 0 & 0 \\
0 & 0 & 1_4 & 0 \\
0 & -1_4 & 0 & 0 \\
0 & 0 & 0 & 1_4
\end{smallmatrix} \right)
\]
where $1_4$ denotes the 4-by-4 identity matrix.

The homomorphism of groups induces a map on coroot lattices (= root lattices for the dual root systems) that describes the restriction of the group homomorphism to Cartan subalgebras on the level of Lie algebras.  Using the concrete description of the group homomorphism above and the choice of Cartan, etc., from \cite{Bou:g7}, we see that the map on coroots is given by Table \ref{ACD}, where the simple roots of $\SL_2, \Sp_8, \Spin_{16}$, and $E_8$ are labelled $\alpha, \gamma, \delta$, and $\e$ respectively and are numbered as in Figures \ref{dynks.fig} and \ref{dynkC.fig}.  
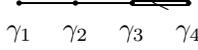
\begin{figure}[ht]
\begin{picture}(5,1.5)
    \put(1,1){\line(1,0){2}} 
    \put(3,1.04){\line(1,0){1}}
    \put(3,0.96){\line(1,0){1}}
 
    \put(1,0.3){\makebox(0,0.5)[b]{$\gamma_1$}}
    \put(2,0.3){\makebox(0,0.5)[b]{$\gamma_2$}}
    \put(3,0.3){\makebox(0,0.5)[b]{$\gamma_3$}}
    \put(4,0.3){\makebox(0,0.5)[b]{$\gamma_4$}}
    
    \put(3.5,0.88){\makebox(0,0.4)[b]{$<$}}



    \multiput(1,1)(1,0){4}{\circle*{\darkrad}}

\end{picture}
\caption{Dynkin diagram of $\rtC$} \label{dynkC.fig}
\end{figure}
For $\SL_2, \Spin_{16}$, and $E_8$, we fix the metric so that roots have length 2, which identifies the coroot and root lattices.  The inclusion of $\hspin_{16}$ in $E_8$ is given by \eqref{d8e8}.\setcounter{table}{\value{figure}}
\begin{table}[ht]
\[
\begin{array}{c@{\ \mapsto\ }c@{\ \mapsto\ }c}
\multicolumn{1}{c}{} & 
\multicolumn{1}{c}{\text{in $D_8$}} & 
\multicolumn{1}{c}{\text{in $E_8$}}\\ \hline
\alpha_1 & \delta_1 + 2 \delta_2 + 3\delta_3 + 4\delta_4 + 3 \delta_5 + 2 \delta_6 + \delta_7 &
-2\e_1 -2\e_2 -4\e_3 -4\e_4 -2\e_5 \\ \hline
\gch_1 & \delta_1 - \delta_7&-2\e_1 -4\e_2 - 4\e_3 - 6\e_4 - 5\e_5 - 4\e_6 -3\e_7 -2\e_8 \\
\gch_2 & \delta_2 - \delta_6 & -\e_4 + \e_8 \\
\gch_3 & \delta_3 -\delta_5 & -\e_5 + \e_7 \\
\gch_4 & \delta_4 + 2\delta_5 + 2\delta_6 + \delta_7 + \delta_8 & \e_2+\e_3 + 2\e_4 + 2\e_5 + \e_6
\end{array} 
\]
\caption{Homomorphisms $\SL_2 \times \Sp_8 \ra \Spin_{16} \ra E_8$ on the level of coroots} \label{ACD}
\end{table}

Either from the explicit tensor product in \eqref{two.iso} or from the description of the center of $\Sp_8$ from \cite[8.5]{GQ}, we deduce inclusions:
\[
(\SL_2 \times \Sp_8)/\mmu2 \subset \Spin_{16} \quad \text{and} \quad \PGL_2 \times \psp_8 \subset \hspin_{16}.
\]
Since the short coroot $\gch_4$ of $\Sp_8$ maps to a (co)root in $\rtD$, the homomorphism $\Sp_8 \ra \Spin_{16}$ has Rost multiplier 1.

(The statements in the previous paragraph can also be deduced from the branching tables in \cite[p.~295]{McKP}, but of course those tables were constructed using data as in Table \ref{ACD}.  To get the statement on Rost multipliers, one uses \cite[7.9]{MG}.)

We find a subgroup $\PGL_2 \times \psp_8 \times \mmu2$ of  $\hspin_{16}$ by taking the center of $\hspin_{16}$ for the copy of $\mmu2$.
\end{borel}

\begin{borel*} \label{dec.inv}
The composition 
\[
H^1(k, \PGL_2 \times \PSp_8 \times \mmu2)  \ra H^1(k, \hspin_{16}) \ra H^1(k, E_8) \xrightarrow{r_{E_8}} H^3(k, \QZt)
\]
defines an invariant of triples $Q, \Cg, c$ where $Q$ is a quaternion algebra, $\Cg$ is a central simple algebra of degree 8 with symplectic involution, and $c$ is in $\kx / k^{\times 2}$.  We abuse notation and write also $r_{E_8}$ for this invariant.

For example, tracing through the proof of Th.~\ref{welldef}, we find:
\begin{equation} \label{rinv.mu2}
r_{E_8}(Q, \Cg, c) = r_{E_8}(Q, (C, \gamma), 1) + (c) \cdot [Q \ot C].
\end{equation}
\end{borel*}

\section{Crux computation}

\begin{lem} \label{rinv.crux}
The composition
\[
H^1(k, \PGL_2) \times 1 \subset H^1(k, \PGL_2) \times H^1(k, \psp_8 \times \mmu2) \ra H^1(k, E_8)
\]
is identically zero.
\end{lem}

We warm up by doing a toy version of a computation necessary for the proof of the lemma.

\begin{eg}\label{descent.quad}
Let $\O_n$ be the orthogonal group of the symmetric bilinear form $f$ with matrix $S_n$ as in \ref{inclusions}.  Fix a quadratic extension $k(\sqrt{a})/k$.  For $\iota$ the nontrivial $k$-automorphism of $k(\sqrt{a})$ and $c \in \kx$, the 1-cocycle $\eta \in Z^1(k(\sqrt{a})/k, \O_2)$ defined by
\[
\eta_\iota = \stbtmat{0}{c}{c^{-1}}{0} 
\]
defines a bilinear form $f_\eta$ over $k$.  It is the restriction of $f$ to the $k$-subspace of $k(\sqrt{a})^2$ of elements fixed by $\eta_\iota \iota$.  This subspace has basis 
\[
\stcolvec{c}{1} \eand \stcolvec{c \sqrt{a}}{-\sqrt{a}},
\]
so $f_\eta$ is isomorphic to $\qform{2c}\qform{1, -a}$.
\end{eg}

\begin{proof}[Proof of Lemma \ref{rinv.crux}]
Fix a cocycle $\eta \in Z^1(k,  \PGL_2)$.  The rank 4 maximal torus in $\psp_8$ (intersection with the torus in $E_8$ specified by the pinning) is centralized by the image of $\eta$, so it gives  a $k$-split torus $S$ in the twisted group $(E_8)_\eta$.  A semisimple anisotropic kernel of $(E_8)_\eta$ is contained in the derived subgroup $D$ of the centralizer of $S$.  The root system of $D$ (over an algebraic closure of $k$) consists of the roots of $\rtE$ orthogonal to the elements of the coroot lattice with image lying in $S$, which are given in Table \ref{ACD}.  The roots of $D$ form a system of rank 4 with simple roots $\phi_1, \phi_2, \phi_3, \phi_4$ as in Table \ref{d4ine8.table};
\begin{table}[ht]
\[
\begin{array}{c|cccc}
D_4&\phi_1&\phi_2&\phi_3&\phi_4 \\ \hline
E_8&\e_6&\e_1+\e_2+2\e_3+2\e_4+\e_5&\e_5+\e_6+\e_7&\e_4+\e_5+\e_6+\e_7+\e_8
\end{array}
\]
\caption{Simple roots in the centralizer of the $C_4$-torus in $E_8$} \label{d4ine8.table}
\end{table}
 they span a system of type $D_4$ with Dynkin diagram as in Figure \ref{d4ine8.fig}, where the unlabeled vertex is the negative $-\phit$ of the highest root $\phi_1 + 2 \phi_2 + \phi_3 + \phi_4 = \et$.
\setcounter{figure}{\value{table}}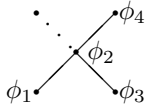
\begin{figure}[bh]
\[
\begin{picture}(3.5,2)
  \put(2,1){\line(-1,-1){0.7}}
    \multiput(1.9,1.1)(-0.2,0.2){3}{\circle*{0.01}}
    \put(1.3,0.3){\circle*{\darkrad}}
    \put(1.3,1.7){\circle*{\darkrad}}
    \put(1,0.1){\makebox(0,0.4)[b]{$\phi_1$}}

    
    \put(2.2,0.8){\makebox(0,0.4)[l]{$\phi_2$}}


    \put(2,1){\line(1,-1){0.7}}
    \put(2,1){\line(1,1){0.7}}

    \put(2,1){\circle*{\darkrad}}

    \put(2.7,0.3){\circle*{\darkrad}}
    \put(2.7,1.7){\circle*{\darkrad}}
    
    \put(3,0.1){\makebox(0,0.4)[b]{$\phi_3$}}
    \put(3,1.5){\makebox(0,0.4)[b]{$\phi_4$}}
    
\end{picture}
\]
\caption{Extended Dynkin diagram of centralizer of the $C_4$-torus in $E_8$} \label{d4ine8.fig}
\end{figure}

We now compute the map $\SL_2 \ra D$.  On the level of tori, it is given by Table \ref{ACD}.  On the level of Lie algebras, we compute using the explicit map \eqref{two.iso} that
the element $\stbtmat{0}{1}{0}{0}$ of $\mathfrak{sl}_2$ maps to
\[
\tbtmat{0}{1}{0}{0} \mapsto 
\left( \begin{smallmatrix}
0 & 0 & 1_4 & 0 \\
0 & 0 & 0 & 1_4 \\
0 & 0 & 0 & 0 \\
0 & 0 & 0 & 0
\end{smallmatrix} \right)
\mapsto
\left( \begin{smallmatrix}
0 & 1_4 & 0 & 0 \\
0 & 0 & 0 & 0 \\
0 & 0 & 0 & -1_4 \\
0 & 0 & 0 & 0
\end{smallmatrix} \right)
\]
in $M_{16}(k)$.  In terms of the Chevalley basis of the Lie algebra of $E_8$, $\stbtmat{0}{1}{0}{0}$ maps to 
\[
X_{\phi_1} + X_{\phi_3} + X_{\phi_4} + X_{-\phit}.
\]

The cocycle $\eta$ represents a quaternion algebra $(a, b)$ over $k$.  If $Q$ is split, then $\eta$ is zero and we are done.  Otherwise, $a$ is not a square.  Replacing $\eta$ with an equivalent cocycle, we may assume that $\eta$ belongs to $Z^1(k(\sqrt{a})/k, \PGL_2)$ and takes the value
\[
\eta_\iota = \Int \stbtmat{0}{b}{1}{0} = \Int \stbtmat{0}{-\sqrt{-b}}{1/\sqrt{-b}}{0}
\]
on the non-identity $k$-automorphism $\iota$ of $k(\sqrt{a})$.  That is, $\eta_\iota$ is conjugation by the element $w_{\alpha_1}(-\sqrt{-b})$ in Steinberg's notation for generators of a Chevalley group from \cite{St}.  The image of $\eta$ in $E_8$ is the cocycle $\ehat$ with
\begin{equation} \label{eta.im}
\ehat_\iota :=  \prod_\phi w_\phi(-\sqrt{-b})
\end{equation}
where $\phi$ ranges over the set $\Sigma := \{ -\phit, \phi_1, \phi_3, \phi_4 \}$.
(The order of terms in the product does not matter, as the roots in $\Sigma$ are pairwise orthogonal.)  We compute the action of this element on each $x_\phi \!: \Ga \ra D_4$ for $\phi \in \Sigma$.  Using the orthogonality of the roots in $\Sigma$, we have:
\begin{equation} \label{eta.act}
\Int(\ehat_\iota) x_\phi(u) = \Int(w_\phi(-\sqrt{-b})) x_\phi(u) = x_\phi(u/b),
\end{equation}
where the second equality is by the identities in \cite[p.~66]{St}.

We now identify $D$ (over an algebraic closure) with $\Spin_8$ using the pinning of $\Spin_8$ from \cite{Bou:g7} and project $\ehat$ to a 1-cocycle with values in $\SO_8$. This cocycle defines a quadratic form $q$ that we claim is hyperbolic.  Indeed, from equation \eqref{eta.act}, we deduce that the image of $\ehat$ in $\SO_8$ is the matrix
\[
\left( \begin{smallmatrix}
&&&&&&&-1 \\
&&&&&&1/b \\
&&&&&-b \\
&&&&1 \\
&&&1 \\
&&-1/b\\
&b\\
-1
\end{smallmatrix} \right)
\]
This preserves the hyperbolic planes in $k^8$ spanned by the 1st and 8th, 2nd and 7th, etc., standard basis vectors, so we can compute the quadratic form by restricting to each of these planes as in Example \ref{descent.quad}.  One finds that $q$ is isomorphic to
\[
\qform{2} \ot \qform{-1, b^{-1}, -b, 1} \ot \qform{1, -a},
\]
which is hyperbolic because the middle term is.  In particular, the twisted group $(\SO_8)_\ehat$ is split, and the same is true for $D$.   We conclude that $(E_8)_\eta$ is split and the image of $\eta$ in $H^1(k, E_8)$ is zero.
\end{proof}

\begin{rmk}
In the case where $k$ contains a square root of every element of the prime field $F$, one can given an easier proof of Lemma \ref{rinv.crux} as follows.  Repeat the first paragraph.  The composition
\[
H^1(*, \PGL_2) \ra H^1(*, E_8) \xrightarrow{r_{E_8}} H^3(*, \QZt)
\]
gives a normalized invariant of $\PGL_2$, which is necessarily of the form $Q \mapsto [Q] \cdot x$ for some fixed $x \in H^1(F, \Zm2)$ by \cite[18.1, \S23]{Se:ci}.  Thus every element of $H^1(k, E_8)$ coming from $H^1(k, \PGL_2)$ has Rost invariant zero (because $x$ is killed by $k$) and is isotropic (obviously), hence is zero by Prop.~\ref{index}(1) below.
\end{rmk}

\section{Rost invariant}

\begin{thm} \label{const.rost}
The composition
\[
H^1(k, \PGL_2) \times H^1(k, \psp_8) \times H^1(k, \mmu2) \ra H^1(k, E_8) \xrightarrow{r_{E_8}} H^3(k, \QZt)
\]
is given by
\[
Q, (C, \gamma), c \mapsto \D(C, \gamma) + (c) \cdot [Q \ot C] \quad \in H^3(k, \Zm2).
\]
\end{thm}

Here $\D$ refers to the discriminant of symplectic involutions on algebras of degree 8 defined in \cite{GPT}.

\begin{proof}
By \eqref{rinv.mu2}, it suffices to prove the case where $c = 1$.

\smallskip

\noindent{\underline{\emph{Step 1.}}} We first verify the proposition in case $C$ has index at most 2 and $\gamma$ is hyperbolic; we write $\Cg = (Q' \ot M_4(k), \binv \ot \hyp)$ for some quaternion algebra $Q'$.  In this way, we restrict $r_{E_8}$ to an invariant $H^1(k, \PGL_2) \times H^1(k, \PGL_2) \ra H^3(k, \QZt)$, which we claim is zero.

We argue as in \cite[\S17]{Se:ci}.  We view $H^1(k, \PGL_2)$ as the image of $H^1(k, \Zm2 \times \Zm2)$ via the map that sends elements $a, b \in \ksq$ to the quaternion algebra $(a, b)$.  By restriction, $r_{E_8}$ can be viewed as an invariant of $(\Zm2)^{\times 4}$; its image consists of elements killed by a quadratic extension (by Lemma \ref{rinv.crux}), so belong to the 2-torsion in $H^3(k, \Zm4)$, which is $H^3(k, \Zm2)$ by \cite{MS:Kcoh}.  Because the image of $H^1(k, (\Zm2)^{\times 4})$ lies in $H^3(k, \Zm2)$ and the value of $r_{E_8}$ on an element is unaltered if we interchange the first two coordinates (corresponding to the quaternion algebra $Q$) or the third and fourth coordinates (corresponding to $Q'$), we deduce that $r_{E_8}$ is of the form
\[
Q, Q' \mapsto \la_0 + \la_Q \cdot [Q] + \la_{Q'} \cdot [Q'] 
\]
for uniquely determined elements $\la_0, \la_Q, \la_{Q'}$ in $H^\bullet(k, \Zm2)$.  (There is no term involving $[Q] \cdot [Q']$ because such a term would have degree at least 4.)

The element $\la_0$ is zero, because the Rost invariant $r_{E_8}$ is normalized.  The coefficient $\la_Q$ is zero by Lemma \ref{rinv.crux}.

Write $K$ for the field obtained by adjoining indeterminates $a, b$ to $k$ and $\Qgen$ for the generic quaternion algebra $(a, b)$ over $K$.  On the one hand,  the value of $r_{E_8}$ on the triple
\[
\Qgen, (\Qgen \ot M_4(k), \binv \ot \hyp), 1 \quad \text{is} \quad 
\la_{Q'} \cdot [\Qgen].
\]
On the other hand, the algebra with involution 
\[
\As := (\Qgen, \binv) \ot (\Qgen \ot M_4(k), \binv \ot \hyp)
\]
has $A$ split and $\s$ hyperbolic, so $\esix\As$ is zero as an element of $H^3(k, \Zm4)$.  Thus $\la_{Q'} = 0$ and $r_{E_8}$ sends $Q, Q'$ to 0.  This verifies  the proposition for $C$ of index at most 2 and $\gamma$ hyperbolic.

\medskip

\noindent{\underline{\emph{Step 2.}}} For a fixed quaternion algebra $Q$, the map 
\[
(C, \gamma) \mapsto r_{E_8}(Q, (C, \gamma), 1) \quad \in H^3(K, \Zm2)
\]
defines an invariant of $H^1(k, \psp_8)$ that is zero on the trivial class by Step 1, hence by \cite[4.1]{GPT} is of the form
\begin{equation} \label{const.1}
\Cg \mapsto \la_1 \cdot [C] + \la_0 \cdot \D\Cg
\end{equation}
for uniquely determined elements $\la_i \in H^i(k, \Zm2)$ for $i = 0, 1$ (which may depend on $Q$).

In case $C$ has index 2 and $\gamma$ is hyperbolic, $\Cg$ maps to zero by Step 1 and $\D\Cg = 0$, so $\la_1 = 0$. 

We are left with deciding whether the element $\la_0$ is 0 or 1 in $H^0(k, \Zm2)$.
Suppose that $C$ has index 2 and put $\eta$ for the image of $Q, (C, \hyp), 1$ in $H^1(k, E_8)$.  The homomorphism
\[
\Sp(C, \hyp) = (\Sp_8)_\eta \ra (E_8)_\eta
\]
has Rost multiplier 1 by \ref{inclusions}, i.e., the induced map
\[
H^1(k, \Sp(C, \hyp)) \ra H^1(k, (E_8)_\eta) \xrightarrow{r_{E_8}} H^3(k, \QZt)
\]
is the Rost invariant of $\Sp(C, \hyp)$.  In particular it is not zero, so $\la_0$ is not zero, i.e., $\la_0  = 1$.  This proves the proposition.
\end{proof}

\section{Value of $\esix$ on decomposable involutions} \label{e3dec.sec}

Let $Q$ be a quaternion algebra and let $\Cg$ be a central simple algebra of degree $8$ with symplectic involution.   The tensor product $(Q, \binv) \ot \Cg$ has degree 16 (clearly),  trivial discriminant \cite[7.3(5)]{KMRT}, and one component of the even Clifford algebra is split \cite[4.15, 4.16]{Tao:Cl}, so the tensor product belongs to $I^3$.   
\begin{cor}[of Th.~\ref{const.rost}] \label{dec.thm}
We have:
\[
\esix \left[ (Q, \binv) \ot \Cg \right] = \D (C, \gamma) \quad\in H^3(k, \Zm4) / [Q \ot C] \cp H^1(k, \mmu2).\hfill\qed
\]
\end{cor}

\begin{borel*}
We compare the invariants $\esix$ (from \S\ref{e3.sec}) and $\ehyp$ (from \S\ref{e3p.sec}).  The invariant $\esix$ is only defined on $\As$ in $I^3$ where $A$ has degree 16.  For such $\As$, the invariant $\ehyp$ is only defined if $A$ is isomorphic to $M_2(C)$ for a central simple algebra $C$ of degree 8.  It turns out that the two invariants agree if the algebra $C$ is decomposable.
\end{borel*}

\begin{cor} \label{agree}
If $\As$ is in $I^3$ and $A$ is isomorphic to $M_2(Q_1 \ot Q_2 \ot Q_3)$ for quaternion algebras $Q_1, Q_2, Q_3$ (e.g., this holds if the index of $A$ is $\le 4$),
then 
$\esix\As = \ehyp\As$.
\end{cor}

\begin{proof} 
The hypothesis on $A$ implies that $A$ supports a hyperbolic involution $\hyp$:
\[
(A, \hyp) \cong (M_2(k), \binv) \ot (C, \gamma) \quad \text{for $(C, \gamma) = \ot_{i=1}^3 (Q_i, \binv)$.}
\]
Let $\eta, \eta' \in H^1(k, \hspin_{16})$ have image $(A, \hyp), \As$ in $H^1(k, \PSO_{16})$ respectively.  The bottom row of \eqref{twist.diag} can be rewritten as
\[
H^1(k, \hspin(A, \hyp)) \ra H^1(k, (E_8)_\eta) \xrightarrow{r_{(E_8)_\eta}} H^3(k, \QZt).
\]
For $\nu \in H^1(k, \Spin(A, \hyp))$ mapping to $\tau(\eta') \in H^1(k, \hspin(A, \hyp))$, we have:
\[
r_{(E_8)_\eta}(\tau(\eta')) = r_{\Spin(A, \hyp)}(\nu) \quad \text{by \eqref{rostmult}.}
\]
The equality
\[
\ehyp\As = \esix\As - \esix(A, \hyp),
\]
follows by commutativity of \eqref{twist.diag}.  Since $\Cg$ is completely decomposable, $\D\Cg$ is zero and $\esix(A, \hyp) = 0$ by Cor.~\ref{dec.thm}, which proves the corollary.
\end{proof}

\part{Groups of type $E_8$ constructed from 9 parameters} 

\section{Construction of $E_8$'s}\label{const.sec}

In \ref{inclusions}, we gave concrete descriptions of embeddings $\PGL_2 \times \psp_8 \times \mmu2 \subset \hspin_{16} \subset E_8$.  Similarly, we can give an explicit embedding $\PGL_2^{\times 3} \subset \psp_8$ as in \cite[Table 9]{Dynk:ssub}.  On the level of coroot lattices the total inclusion 
\begin{equation} \label{const.emb}
\PGL_2^{\times 4} \times \mmu2 \subset \PGL_2 \times \psp_8 \times \mmu2 \subset \hspin_{16} \subset E_8
\end{equation}
is described in Table \ref{PGL3}.  We remark that the four copies of $\PGL_2$ are not normalized by a maximal torus of $E_8$.\begin{table}[ht]
\begin{tabular}{|ccc|} \hline
\parbox{1in}{simple (co)root in copy of $\PGL_2$}&in $C_4$&in $E_8$ \\ \hline 
$\alpha_1$&& $-(2\e_1 + 2\e_2 + 4\e_3 + 4\e_4 + 2\e_5)$\\
$\alpha_2$&$\gch_1 - \gch_3$& $-(2\e_1 + 4 \e_2 + 4 \e_3 + 6\e_4 + 4 \e_5 + 4\e_6 + 4 \e_7 + 2 \e_8)$\\
$\alpha_3$&$\gch_1 + \gch_3$& $-(2\e_1 + 4 \e_2 + 4 \e_3 + 6\e_4 + 6 \e_5 + 4\e_6 + 2 \e_7 + 2 \e_8)$\\
$\alpha_4$&$\gch_1 + 2\gch_2 + \gch_3$& $-(2\e_1 + 4 \e_2 + 4 \e_3 + 8\e_4 + 6 \e_5 + 4\e_6 + 2 \e_7)$ \\ \hline
\end{tabular}
\caption{Inclusion $\PGL_2^{\times 4} \subset \PGL_2 \times \psp_8 \subset E_8$ on the level of coroots} \label{PGL3}
\parbox{4in}{$\alpha_i$ is a simple (co)root in the $i$-th copy of $\PGL_2$.  That copy is inside $\psp_8$ for $i \ne 1$.}
\end{table}

Applying Galois cohomology to \eqref{const.emb} gives a function
\begin{equation} \label{const.map}
H^1(k, \PGL_2^{\times 4} \times \mmu2) \ra H^1(k, E_8).
\end{equation}
The first set classifies quadruples $(Q_1, Q_2, Q_3, Q_4)$ of quaternion $k$-algebras together with an element $c \in \ksq$, and the second set classifies groups of type $E_8$ over $k$.  Therefore we may view the function \eqref{const.map} as a construction of groups of type $E_8$ via Galois descent.  

\begin{cor}[of Th.~\ref{const.rost}] \label{e8const.rost}
The Rost invariant of a group of type $E_8$ constructed from $(Q_1, Q_2, Q_3, Q_4, c)$ is $(c) \cp \sum [Q_i]$.
\end{cor}

\begin{proof}
As $\Cg$ is the tensor product $\ot_{i=2}^4 (Q_i, \binv)$, it is decomposable, and so has discriminant zero \cite{GPT}.  Th.~\ref{const.rost} gives the claim.
\end{proof}

\begin{borel*} \label{tensor3}
How much can we vary the data $(Q_1, Q_2, Q_3, Q_4, c)$ without changing the resulting group of type $E_8$?  For example, let $Q'_2, Q'_3, Q'_4$ be quaternion algebras such that the tensor products $\ot_{i=2}^4 (Q'_i, \binv)$ and $\ot_{i=2}^4 (Q_i, \binv)$ are isomorphic as algebras with involution.  Then the images of $(Q_1, Q_2, Q_3, Q_4, c)$ and $(Q_1, Q'_2, Q'_3, Q'_4, c)$ in $H^1(k, \PGL_2) \times H^1(k, \psp_8) \times H^1(k, \mmu2)$ agree, hence one obtains the same group of type $E_8$ from the two inputs.
\end{borel*}

We also have:

\begin{prop} \label{Q.perm}
For every permutation $\pi$, the group of type $E_8$ constructed from $(Q_{\pi 1}, Q_{\pi 2}, Q_{\pi 3}, Q_{\pi 4}, c)$ is the same.
\end{prop}

\begin{proof}
We compare the images $\eta$ and $\eta_\pi$ of $(Q_1, Q_2, Q_3, Q_4, c)$ and  $(Q_{\pi 1}, Q_{\pi 2}, Q_{\pi 3}, Q_{\pi 4}, c)$ respectively in $H^1(k, \hspin_{16})$.  As both $\eta$ and $\eta_\pi$ map to the class of $\otimes (Q_i, \binv)$ in $H^1(k, \PSO_{16})$, the class of $\eta_\pi$ is $\zeta_\pi \cdot \eta$ for some $\zeta_\pi \in H^1(k, \mmu2)$.  The element $\zeta_\pi$ is uniquely determined as an element of the abelian group $\G := H^1(k, \mmu2) / \im (\PSO_{16})_\eta(k)$, where $(\PSO_{16})_\eta(k)$ maps into $H^1(k, \mmu2)$ via the connecting homomorphism arising from the exact sequence at the bottom of diagram \eqref{hspin.diag}, see \cite[\S{I.5.5}, Cor.~2]{SeCG}.  This defines a homomorphism $\zeta$ from the symmetric group on 4 letters, $\sym_4$, to $\G$.

As $\G$ is abelian, the homomorphism $\zeta$ factors through the commutator subgroup of $\sym_4$, the alternating group.  But $\zeta$ vanishes on the odd permutation $(3 \, 4)$ by \ref{tensor3}, so $\zeta$ is the zero homomorphism.  This proves the proposition.
\end{proof}

\begin{borel}{Tits's construction} \label{Tits.const}
In \cite{Ti:const}, Tits gave a construction of algebraic groups of type $E_8$ with inputs an octonion algebra and an Albert algebra.  In terms of algebraic groups, there is an (essentially unique) inclusion of $G_2 \times F_4$ in $E_8$ \cite[p.~226]{Dynk:ssub}, and Tits's construction is the resulting map in Galois cohomology:
\[
H^1(k, G_2) \times H^1(k, F_4) \ra H^1(k, E_8).
\]
His construction and ours from \eqref{const.map} overlap, but they are distinct.

We compute the Rost invariant of a group $G$ of type $E_8$ constructed by Tits's recipe from an octonion algebra with 3-Pfister norm form $\gamma_3$ and an Albert algebra $A$.  Because the inclusions $G_2 \subset E_8$ and $F_4 \subset E_8$ both have Rost multiplier 1 \cite[p.~192]{Dynk:ssub}, we have:
\[
r_{E_8}(G) = e_3(\gamma_3) + r_{F_4}(A),
\]
where $r_{F_4}$ denotes the Rost invariant relative to the split group of type $F_4$.  Associated with $A$ are Pfister forms $\phi_3$ and $\phi_5$, where $\phi_i$ has dimension $2^i$ and $\phi_3$ divides $\phi_5$, see \cite[22.5]{Se:ci}.  We find:
\[
15r_{E_8}(G) = e_3(\gamma_3 + \phi_3) \quad \in H^3(k, \Zm2).
\]
\end{borel}

\section{Tits index of groups of type $E_8$} \label{index.sec}

In this section, we note some relationships between the Tits index of a group $G$ of type $E_8$ over $k$ and its Rost invariant $r_{E_8}(G)$.  

Recall that if $r_{E_8}(G)$ is killed by a quadratic extension or is 2-torsion, then it belongs to $H^3(k, \Zm2)$.  A \emph{symbol} is an element of the image of the cup product map $H^1(k, \Zm2)^{\times 3} \ra H^3(k, \Zm2)$.  The \emph{symbol length} of an element $x \in H^3(k, \Zm2)$ is the smallest integer $n$ such that $x$ is equal to a sum of $n$ symbols in $H^3(k, \Zm2)$.  Zero is the unique element with symbol length 0.

\begin{prop} \label{index}
Let $G$ be an isotropic group of type $E_8$.  Then:
\begin{enumerate}
\item If $r_{E_8}(G)$ is zero, then $G$ is split.
\item If $r_{E_8}(G)$ is split by a quadratic extension of $k$, then $r_{E_8}(G)$ has symbol length $\le 3$ in $H^3(k, \Zm2)$.
\item If $r_{E_8}(G)$ is $2$-torsion and $G$ has $k$-rank $\ge 2$, then the Tits index is given by Table \ref{E7.index.table}.
\end{enumerate}
\end{prop}

\begin{table}[ht]
\begin{tabular}{|c|c|} \hline
index&symbol length of $r_{E_8}(G)$\\ \hline
split&$0$ \\ 
\begin{picture}(6.5,1.3)
    \multiput(0.3,.3)(1,0){7}{\circle*{\darkrad}}
    \put(2.3,1){\circle*{\darkrad}}

    \put(0.3,0.3){\line(1,0){6}}
    \put(2.3,1){\line(0,-1){0.7}}
    
    \put(0.3,0.3){\circle{\lrad}}
    \multiput(6.3,.3)(-1,0){3}{\circle{\lrad}}
    
    \end{picture}&$1$ \\
\begin{picture}(6.5,1.3)
    \multiput(0.3,.3)(1,0){7}{\circle*{\darkrad}}
    \put(2.3,1){\circle*{\darkrad}}

    \put(0.3,0.3){\line(1,0){6}}
    \put(2.3,1){\line(0,-1){0.7}}
    
    \put(0.3,0.3){\circle{\lrad}}
    \multiput(6.3,.3)(-1,0){1}{\circle{\lrad}}
    
    \end{picture}&$2$ \\ \hline
\end{tabular}
\caption{Tits index versus symbol length for isotropic groups of type $E_8$ such that $r_{E_8}(G)$ is 2-torsion} \label{E7.index.table}
\end{table}

\begin{borel*} \label{index.context}
Before proving the proposition, we give some context for it.  We consider a group $G$ constructed via \eqref{const.map} from quaternion algebras $Q_1, Q_2, Q_3, Q_4$ and some $c \in \ksq$.  If at least one of the $Q_i$ is split, then $G$ contains a subgroup isomorphic to $\PGL_2$ and so is isotropic.  If at least two of the $Q_i$ are split or a tensor product of some three of them is split, then $G$ contains a subgroup isomorphic to $\PGL_2 \times \PGL_2$ or $\psp_8$ respectively, and so has $k$-rank $\ge 2$.  In any case, 
the Rost invariant of $G$ is zero over $k(\sqrt{c})$, which is either $k$ or a quadratic extension of $k$. 
\end{borel*}

In particular, if one wishes to use \eqref{const.map} to construct groups of type $E_8$ that are non-split but in the kernel of the Rost invariant, then none of the $Q_i$ can be split, nor can any tensor product of three of them.

\begin{eg}
If the quaternion algebras $Q_2, Q_3, Q_4$ are split and $(c) \cdot [Q_1] \ne 0$.  Then $(c) \cdot [Q_1]$ is a symbol in $H^3(k, \Zm2)$ corresponding to a 3-Pfister form $q$.  By Proposition \ref{index}(3), $G$ has semisimple anisotropic kernel $\Spin(q)$.
\end{eg}

\begin{eg} \label{symb.4}
For ``generic" $c$ and quaternion algebras $Q_i$, construction \eqref{const.map} gives a group $G$ of type $E_8$ whose Rost invariant is killed by a quadratic extension of $k$ and has symbol length 4.  The group $G$ is anisotropic by Prop.~\ref{index}.
\end{eg}

We prepare the proof of Prop.~\ref{index} with lemmas on groups of type $D_6$ and $E_7$.

\begin{lem} \label{Spin12}
The Witt index of a $12$-dimensional quadratic form $q \in I^3$ is given by the table:
\[
\begin{tabular}{|c|ccc|} \hline
Witt index of $q$&$0$&$2$&$6$ \\ \hline
symbol length of $e_3(q)$ in $H^3(k, \Zm2)$&$2$&$1$&$0$ \\ \hline
\end{tabular}
\]
\end{lem}

\begin{proof}
The Witt index of $q$ cannot be 1 because 10-dimensional forms in $I^3$ are isotropic.  Similarly, it cannot be 3, 4, or 5 by the Arason-Pfister Hauptsatz.  This shows that 0, 2, and 6 are the only possibilities.

The form is hyperbolic if and only if it belongs to $I^4$, i.e., $e_3(q)$ is zero; this proves the last column of the table.  If the Witt index is 2, then $q = \qform{c} \gamma \oplus 2 \H$ for some $c \in \kx$ and anisotropic 3-Pfister form $\gamma$, so $e_3(q)$ is a symbol.  Finally, suppose that $e_3(q)$ has symbol length 1, i.e., $q - \gamma$ is in $I^4$ for some anisotropic 3-Pfister form $\gamma$.  Over the function field $K$ of $\gamma$, the form $q$ is hyperbolic by Arason-Pfister, so $q = \qform{c} \gamma \oplus 2 \H$ for some $c \in \kx$ by \cite[X.4.11]{Lam}.
\end{proof}

In the next lemma, we write $E_7$ for the split simply connected group of that type, and $E_6^K$ for the quasi-split simply connected group of type $E_6$ associated with a quadratic \'etale $k$-algebra $K$.

\begin{lem} \label{E6E7}
There is an inclusion of $E^K_6$ in $E_7$ with Rost multiplier $1$ such that the induced map $H^1(K/k, E^K_6) \ra H^1(K/k, E_7)$ is surjective.
\end{lem}

A class $\eta \in H^1(k, E_7)$ is split by $K$ if and only if $K$ kills the Rost invariant $r_{E_7}(\eta)$ by \cite{G:rinv}.  It follows from \cite[3.6]{G:rinv} that there is some quadratic \'etale $k$-algebra $L$ such that $\eta$ is in the image of $H^1(K/k, E^L_6) \ra H^1(K/k, E_7)$.  The point of the lemma is to arrange that $L = K$.

\begin{proof}[Proof of Lemma \ref{E6E7}]
We view $E_7$ as the identity component of the group preserving a quartic form on the 56-dimensional vector space $\stbtmat{k}{J}{J}{k}$ for $J$ the split Albert algebra, cf.~\cite{G:struct}.  Write $S$ for the subgroup of $E_7$ that stabilizes the subspaces $\stbtmat{k}{0}{0}{0}$ and $\stbtmat{0}{0}{0}{k}$; it is reductive and its derived subgroup is the split simply connected group of type $E_6$.

As in \cite[3.5]{G:rinv}, for $i$ a square root of $-1$, the map
\[
\stbtmat{\alpha}{x}{y}{\beta} \mapsto \stbtmat{i\beta}{iy}{ix}{i\alpha}
\]
is a $k(i)$-point of $E_7$, and this gives an inclusion $\mmu4 \injects E_7$.  Twisting $E_7$ by a 1-cocycle $\nu \in Z^1(k, \mmu4)$ that maps to the class of $K$ in $H^1(k, \mmu2)$, we find inclusions
\[
E^K_6 = (E_6)_\nu \subset S_\nu \subset (E_7)_\nu \cong E_7.
\]
Write $\iota$ for the nontrivial $k$-automorphism of $K$.  The group $S_\nu$ is the intersection $P \cap \iota(P)$ for a maximal parabolic $K$-subgroup $P$ of $(E_7)_\nu$, hence the map
$H^1(K/k, S_\nu) \ra H^1(K/k, E_7)$ is surjective \cite[pp.~369, 383]{PlatRap}.

We have an exact sequence
\[
\begin{CD}
1 @>>> E^K_6 @>>> S_\nu @>\pi>> R^1_{K/k}(\Gm) @>>> 1
\end{CD}
\]
where $\pi$ is the map that sends the endomorphism
\[
\stbtmat{\alpha}{x}{y}{\beta} \mapsto \stbtmat{\mu \alpha}{f(x)}{f^\dagger(y)}{\mu^{-1} \beta}
\]
to $\mu$.  Formula \cite[1.6]{G:struct} gives an explicit splitting $s$ of $\pi$ defined over $k$ such the image of $s$ is contained in the parabolic subgroup $Q$ of $(E_6)_\nu$ stabilizing the  subspace $W$ from \cite[6.8]{G:struct}.

Fix an element $\eta' \in H^1(K/k, S_\nu)$ that maps to $\eta \in H^1(K/k, E_7)$.  We twist $S_\nu$ by $s\pi(\eta')$.  The image of $\eta''$ of $\eta'$ in $H^1(K/k, (S_\nu)_{s\pi(\eta')})$ under the twisting map takes values in the semisimple part $D := ((E_6)_\nu)_{s\pi(\eta')}$.  But $D$ contains the $k$-parabolic  $Q$, hence $D$ is quasi-split or has semisimple anisotropic kernel a transfer $R_{K/k}(H)$ where $H$ is anisotropic of type $^1\!A_2$.  But this is impossible because $D$ is split by a quartic extension of $k$, so $D$ is the quasi-split group $E^K_6$.
\end{proof}

\begin{proof}[Proof of Prop.~\ref{index}]
Statement (1) is standard, so we only sketch the proof.  The semisimple anisotropic kernel of an isotropic but non-split group of type $E_8$ is a simply connected group of type $E_7, D_7, E_6, D_5$, or $D_4$ by \cite[p.~60]{Ti:Cl}, but the Rost invariant has zero kernel for a split group of that type \cite{G:rinv}.  Statement (1) now follows by Tits's Witt-type theorem \cite[2.7.2(d)]{Ti:Cl}.

For (3), we may assume that $G$ is not split, equivalently that $r_{E_8}(G)$ has positive symbol length.  Because the $k$-rank of $G$ is at least 2, Tits's table in \cite[p.~60]{Ti:Cl} shows that the semisimple anisotropic kernel of $G$ is a strongly inner simply connected group of type $E_6, D_6$, or $D_4$.  The first case is impossible because $r_{E_8}(G)$ is 2-torsion.  Statement (3) now follows from Tits's Witt-type theorem and Lemma \ref{Spin12}.

To prove (2), by (3) we may assume that $G$ has $k$-rank 1, hence that the semisimple anisotropic kernel of $G$ is a strongly inner simply connected group of type $D_7$ or $E_7$.  In the first case, $r_{E_8}(G)$ is the Arason invariant of a 14-dimensional form in $I^3$, hence has symbol length $\le 3$ by \cite[Prop.~2.3]{HT}.  For the second case, by Lemma \ref{E6E7} it suffices to prove that the Rost invariant of every element of $H^1(K/k, E^K_6)$ has symbol length at most 3, which is \cite[p.~321]{Ch:rinv}.
\end{proof}

Presumably the methods of \cite{Ch:rinv} can be used to give an alternative proof of Prop.~\ref{index}(2) that avoids Lemma \ref{E6E7}.

\section{Reduced Killing form up to Witt-equivalence} \label{kappa.sec}

Recall that the \emph{reduced Killing form} of $G$ --- which we denote by $\rkill_G$ --- is equal to the usual Killing form divided by twice the dual Coxeter number \cite[\S5]{GrossNebe}.  For a group $G$ of type $E_8$, all the roots of $G$ have the same length and the dual Coxeter number equals the (usual) Coxeter number, which is 30.  Hence the usual Killing form $\kill_G$ satisfies $\kill_G = 60\, \rkill_G$ and $\kill_G$ is zero in characteristics 2, 3, 5.

We identify the bilinear form $\rkill_G$ with the quadratic form (and element of the Witt ring) $x \mapsto \rkill_G(x,x)$.

\begin{eg}[the split group] \label{split.rkill}
The reduced Killing form for the split group $E_8$ of that type is Witt-equivalent to the 8-dimensional form $\qform{1, 1, \ldots, 1}$, which we denote simply by 8.  To see this, note that the positive root subalgebras span a totally isotropic subspace parallel to the isotropic subspace spanned by the negative root subalgebras, so $\rkill_{E_8}$ is Witt-equivalent to its restriction to the Lie algebra of a split maximal torus.  By \cite{GrossNebe}, this restriction is isomorphic to the quadratic form $x \mapsto xCx^t$ for $C$ the Cartan matrix of the root system, and it is easy to check that this quadratic form is 8.
\end{eg}

\begin{eg}[Tits's groups] \label{Tits.rkill}
For a group $G$ of type $E_8$ obtained from Tits's construction as in \ref{Tits.const}, we have
\[
\rkill_G = \qform{2} \left[ 8 - \left( 4 \gamma_3 + 4 \phi_3 + \qform{2} \gamma_3 (\phi_5 - \phi_3)  \right) \right]
\]
by \cite[p.~117, (144)]{Jac:ex}, where the Killing forms for the subalgebras of type $F_4$ and $G_2$ are given in \cite[27.20]{Se:ci}.
\end{eg}

\begin{borel*} \label{kappa.def}
The map $G \mapsto \qform{2} (\rkill_{E_8} - \rkill_G)$ defines a Witt-invariant of $H^1(*, E_8)$ in the sense of \cite[\S27]{Se:ci}, i.e., a collection of maps 
\[
\kappa \!: \fbox{\parbox{1.5in}{groups of type $E_8$ over $K$}} \ra W(K)
\]
for every extension $K/k$ (together with some compatibility condition), where $W(K)$ denotes the Witt ring of $K$. 
\end{borel*}

\begin{eg}[groups over $\R$] \label{real.eg0}
For each of the three groups of type $E_8$ over the real numbers, we list the Rost invariant, the (signature of the) Killing form,  and the value of $\kappa$.  All three groups are obtained by Tits's construction \cite[p,~121]{Jac:ex}, so the Killing form and Rost invariant are provided by the formulas in \ref{Tits.rkill} and \ref{Tits.const}.
The Rost invariant $r_{E_8}$ takes values in $H^3(\R, \QZt) = H^3(\R, \Zm2) = \Zm2$.
\[
\begin{tabular}{|c|ccc|} \hline
$G$&$r_{E_8}(G)$&$\kill_G$&$\kappa(G)$\\ \hline
split&0&8&0 \\
other&1&$-24$&$32 \in I^5$\\
anisotropic/compact&0&$-248$&$256 \in I^8$\\ \hline
\end{tabular}
\]
\end{eg}

The rest of this section is concerned with the question over the field $k$: \emph{In what power of $I$ does $\kappa(G)$ lie?}

\begin{lem} \label{inI3}
For every group $G$ of type $E_8$, $\kappa(G)$ belongs to $I^3$.
\end{lem}
\begin{proof}
Because $E_8$ is simply connected, the adjoint representation $E_8 \ra \O(\rkill_{E_8})$ lifts to a homomorphism $E_8 \ra \Spin(\rkill_{E_8})$.  For $G$ a group of type $E_8$ over an extension $K/k$, 
the reduced Killing form of $G$ is the image of $G$ under the composition
\[
H^1(K, E_8) \ra H^1(K, \Spin(\rkill_{E_8})) \ra H^1(K, \O(\rkill_{E_8})),
\]
so it belongs to $I^3 K$, see e.g.~\cite[p.~437]{KMRT}.
\end{proof}

\begin{lem} \label{inI4}
For every group $G$ of type $E_8$, we have $30r_{E_8}(G) = e_3(\kappa(G))$.  The form $\kappa(G)$ is in $I^4$ if and only if $30r_{E_8}(G) = 0$.
\end{lem}

\begin{proof}
We have a commutative diagram, where $\fe_8$ denotes the Lie algebra of $E_8$ and arrows are labeled with Rost multipliers:
\[
\xymatrix{
E_8 \ar[r] \ar[dr]^-{60} & \Spin(\fe_8) \ar[d]^-2 \\
&\SL(\fe_8)
}
\]
Therefore, the Rost multiplier of the top arrow is $60/2 = 30$ \cite[p.~122]{MG}.  The first claim follows.  The second claim amounts to the observation that the kernel of $e_3$ is $I^4$.
\end{proof}

Because the Rost invariant $r_{E_8}$ has order 60 \cite[16.8]{MG}, the value of $r_{E_8}$ on a versal $E_8$-torsor $G$ has order 60 \cite[12.3]{Se:ci}.  
In particular, $\kappa(G)$ does not belong to $I^4$, so Lemma \ref{inI3} cannot be directly strengthened.  

We return to the question ``In what power of $I$ does $\kappa(G)$ lie?" in \S\ref{conj.sec} below.

\section{Calculation of the Killing form} \label{kill.sec}

We now compute the Killing form of a group $G$ of type $E_8$ constructed via \eqref{const.map}.  

\begin{thm}[$\chr k \ne 2, 3$] \label{E8.kill}
Let $G$ be a group of type $E_8$ constructed from quaternion algebras $Q_1, Q_2, Q_3, Q_4$ and $c \in \ksq$ as in \eqref{const.map}. Then the reduced Killing form of $G$ is isomorphic to
\[
8 \qform{2c} - \qform{2}\qform{1, -c} \left( \sum_i Q'_i + \sum_{i < j < \ell} Q'_i Q'_j Q'_\ell \right).
\]
\end{thm}

Here we have written $Q'_i$ for the unique 3-dimensional quadratic form such that $\qform{1} \oplus Q'_i$ is the norm on $Q_i$.

We prove the theorem by restricting the adjoint representation $\g$ of $G$ to the subgroup $\hspin\As$ for $\As = \otimes (Q_i, \binv)$.
We can compute the restriction of $G$ to $\hspin\As$ over an algebraic closure, where we find that $\g$ is a direct sum of the adjoint representation of $\hspin_{16}$ and the natural half-spin representation \cite[p.~305]{McKP}. Both representations of $\hspin_{16}$ are irreducible: the half-spin representation because it is minuscule; for the adjoint representation see \cite[2.6]{St:aut}.  As the reduced Killing form $\rkill_G$ is invariant under $\hspin_{16}$, it follows that these two irreps of $\hspin_{16}$ in $\g$ are orthogonal relative to $\rkill_G$.  We compute the restriction of the reduced Killing form to each summand separately.  

(We use the reduced Killing form instead of the usual one, in order to get a nontrivial result in characteristic 5.  The statement of the theorem makes sense in characteristic 3, but our proof does not work in that case.)

\begin{borel}{Killing form of $\hspin\As$} \label{D8.kill}
The Lie algebra of $\hspin\As$ can be identified with the space of $\s$-skew-symmetric elements in $A$.  Because $\As$ is a tensor product of quaternion algebras, this space can be described inductively as a tensor product of the subspaces of the $Q_i$ that are symmetric or skew-symmetric under $\binv$.  (It is clear that such tensor products can be formed that belong to the skew elements in $A$, and dimension count shows that all skew elements of $A$ are obtained in this way.)  The trace quadratic form $q \mapsto \tr_{Q_i}(q^2)$ restricts to be $\qform{2}$ on the $\binv$-symmetric elements and $\qform{-2} Q'_i$ on the $\binv$-skew-symmetric elements.  We conclude that the form $a \mapsto \tr_A(a^2)$ on $A$ restricts to 
\begin{equation} \label{D8.1}
\qform{-1} \left( \sum_i Q'_i + \sum_{i < j < \ell} Q'_i Q'_j Q'_\ell \right)
\end{equation}
on the $\s$-skew-symmetric elements.

The form \eqref{D8.1} is invariant under $\hspin\As$, so it is a scalar multiple of the Killing form.  By \cite[chap.~VIII, \S13, Exercise 12]{Bou:g7}, the Killing form of $\hspin\As$ is $\qform{h_{D_8}}$ times the form \eqref{D8.1}, where $h_{D_8}$ denotes the Coxeter number, which is 14.  

Note that for $X_\alpha, X_{-\alpha}$ belonging to our fixed pinning of $E_8$ and spanning the highest and lowest root subalgebras of $\hspin_{16}$, we have
$\kill_{\hspin_{16}}(X_\alpha, X_{-\alpha}) = 2h_{D_8}$, but $\kill_{E_8}(X_\alpha, X_{-\alpha}) = 2 h_{E_8}$ by \cite[pp.~E-14, E-15]{SpSt}, where $h_{E_8}$ is the Coxeter number of $E_8$, i.e., 30.  We conclude that the restriction of the Killing form of $E_8$ to the adjoint representation of $\hspin\As$ is $\qform{h_{E_8}}$ times the form \eqref{D8.1}, i.e., the reduced Killing form of $G$ restricts to be $\qform{2}$ times \eqref{D8.1} on the Lie algebra of $\hspin_{16}$.
\end{borel}

\begin{borel}{Half-spin representation} \label{D8.half}
We restrict the half-spin representation of $\hspin\As$ to the product of the $\PGL(Q_i)$'s.  Putting $\omega_i$ for the unique fundamental dominant weight of $\PGL(Q_i)$, the half-spin representation decomposes as a direct sum of the four 5-dimensional representations with highest weight $4\omega_i$ and the four 27-dimensional representations with highest weight $2\omega_i + 2\omega_j + 2\omega_\ell$ with $i < j < \ell$.

Now the representations of $\PGL(Q_i)$ with highest weights $2\omega_i$ and $4\omega_i$ support $\PGL(Q_i)$-invariant quadratic forms isomorphic to 
\[
\qform{2} Q'_i \quad \text{and} \quad \qform{2} Q_i + \qform{6}
\]
respectively by \cite{G:A1}.  (The cited result concerns Weyl modules, but the modules with these highest weights are irreducible in characteristic $\ne 2, 3$.)  Because the $Q_i$ are interchangeable (Prop.~\ref{Q.perm}), the half-spin representation contributes
\begin{equation} \label{D8.2}
\qform{2} \qform{c\, m_2} \sum_{i < j < \ell} Q'_i Q'_j Q'_\ell \oplus \qform{c\, m_4} \sum_i ( \qform{2} Q_i + \qform{6})
\end{equation} 
to the reduced Killing form of $G$, where $m_2, m_4$ are elements of $\kx$ that are not yet determined.  

We first compute $m_4$ up to sign, for which it suffices to consider the case $c = 1$.  As the $Q_i$'s are interchangeable, we focus on $Q_1$.  The only root of $E_8$ that restricts to $4 \omega_1$ is $-(\e_1 + \e_2 + 2\e_3 + 2\e_4 + \e_5)$.  We write $X_1$ for the element of the Chevalley basis of $E_8$ corresponding to that root and $X_{-1}$ for the element corresponding to the negation of the root; $X_1$ and $X_{-1}$ are highest and lowest weight vectors respectively for the irrep of $\PGL_2^{\times 4}$ with highest weight $4\omega_1$.  We have
\[
\stbtmat{0}{1}{-1}{0} X_1 = \pm X_{-1}
\]
because all the roots in $E_8$ have the same length.  Therefore,
\[
\rkill_{E_8}(X_1, \stbtmat{0}{1}{-1}{0} X_1) = \pm \rkill_{E_8}(X_1, X_{-1}) = \pm 1.
\]
On the other hand, for $f$ the symmetric bilinear form on the representation of $\PGL(Q_1)$ with highest weight $4 \omega_1$ such that $f$ is isomorphic to $\qform{2} Q_1 + \qform{6}$, we have
\[
f(X_1, \stbtmat{0}{1}{-1}{0} X_1) = 1,
\]
see \cite[2.4]{G:A1}.
Consequently, $m_4 = \pm 1$.  Similarly, $m_2$ is also 1 or $-1$, possibly with a different sign from $m_4$.
\end{borel}

\begin{borel*}
Combining the results of \ref{D8.kill} and \ref{D8.half}, we find that the reduced Killing form of $G$ is
\begin{equation} \label{D8.mu}
4 \qform{2c\,m_4} \qform{1,3} + \qform{2} \qform{-1, c\,m_4} \sum Q'_i + \qform{2} \qform{-1, c\,m_2} \sum Q'_i Q'_j Q'_\ell.
\end{equation}

As $m_2$ and $m_4$ are $\pm 1$ and are defined over $\Z$, to compute their values in $\ksq$ it suffices to compute their values in $\R^{\times} / \R^{\times 2}$.  Consider the case $k = \R$ and $c = 1$ and where exactly one of the $Q_i$ is nonsplit.  Then 
\[
\sum Q'_i = 0 \eand \sum Q'_i Q'_j Q'_\ell = 8.
\]
Plugging this data into \eqref{D8.mu}, we find that the reduced Killing form is
$8\qform{m_2, m_4} - 8$.
On the other hand, the group is isotropic and has Rost invariant zero, so it is split and the reduced Killing form is 8.  It follows that $\qform{m_2, m_4}$ is 2, i.e., $m_2 = m_4 = 1$.

Returning to the general case, $4\qform{3}$ is isomorphic to $4\qform{1}$.
This completes the proof of Th.~\ref{E8.kill}.$\hfill\qed$.
\end{borel*}

\begin{eg} \label{Pf4}
As a corollary of the proof, we can see that in case $q$ is a 4-Pfister quadratic form, there is a $\Spin(q)$-invariant quadratic form on each half-spin representation that is isomorphic to $8q$.  Indeed, we can write $q$ as a product $q_1 q_2$ of 2-Pfister forms, and let $Q_i$ be the quaternion algebra with norm $q_i$.  The tensor product $(Q_i, \binv) \ot (Q_i, \binv)$ is $M_4(k)$ endowed with an orthogonal involution adjoint to $q_i$.  Setting $Q_3 = Q_1$ and $Q_4 = Q_2$ and evaluating \eqref{D8.2} with $c = \mu_2 = \mu_4 = 1$ gives the claim.
\end{eg}

\section{The Killing form and $E_8$'s arising from \eqref{const.map}}

This section gives some properties of the groups of type $E_8$ constructed via \eqref{const.map}, proved using the calculation of the Killing form in the preceding section.

\begin{eg} \label{num.eg}
The classification of groups of type $E_8$ over the real numbers was recalled in Example \ref{real.eg0}.   Such groups are distinguished (up to isomorphism) by their Killing forms.
We now observe that construction \eqref{const.map} produces all three groups.  There are only two possibilities (split or not) for each of the four quaternion algebras as well as for $c$.
\[
\begin{tabular}{|rp{1.5in}|p{0.75in}c|} \hline
$c$&\# of nonsplit quaternion algebras&signature of Killing form&description of group \\ \hline
1 & 0 through 4 & 8 &split \\
$-1$ & 0 or 2 & 8 & split \\
$-1$&1 or 3 & $-24$ & isotropic \\
$-1$ & 4 & $-248$ & anisotropic/compact \\ \hline
\end{tabular}
\]

Chernousov's Hasse Principle \cite[p.~286, Th.~6.6]{PlatRap} for $E_8$ implies that construction \ref{const.map} produces every group of type $E_8$ over a number field.
\end{eg}

Recall the definition of $\kappa(G)$ from \ref{kappa.def}; it measures the difference of the reduced Killing form of $G$ from the same form for the split $E_8$.

\begin{prop}[$\chr k \ne 2, 3$] \label{E8.witt}
For a group $G$
of type $E_8$ constructed from quaternion algebras $Q_1, Q_2, Q_3, Q_4$ and $c \in \ksq$, we have
\[
\kappa(G) = \pform{c} \left[ 
 4 \sum_i Q_i - 2 \sum_{i<j} Q_i Q_j + \sum_{i<j<\ell} Q_i Q_j Q_\ell 
\right] \quad \in I^5.
\]
\end{prop}

We use the notation $\pform{x_1, \ldots, x_n}$ for the tensor product $\qform{1, -x_1} \ot \cdots \ot \qform{1, -x_n}$.

\begin{proof}[Proof of Prop.~\ref{E8.witt}]
Example \ref{split.rkill} and Th.~\ref{E8.kill} give
\[
\kappa(G) = \pform{c} \left( 8 + \sum_i Q'_i + \sum_{i < j < \ell} Q'_i Q'_j Q'_\ell \right).
\]
Replacing each $Q'_i$ with $Q_i - 1$ and expanding gives the displayed formula.
The form is obviously in $I^5$. 
\end{proof}

\begin{borel*} \label{inv.I5}
It is easy to see that $I^5$ in Prop.~\ref{E8.witt} ``cannot be improved", i.e., that $\kappa(G)$ need not lie in $I^6$.  One can take $k$ to be $\R$ with indeterminates $x, y, c$ adjoined, and put $Q_1 := (x, y)$ and $Q_2, Q_3, Q_4$ split.  The resulting $G$ has
$\kappa(G) = 4 \pform{c, x, y}$.

Similarly, if $-1$ is a square in $k$, then $2 = 0$ in the Witt ring and $\kappa(G) = \pform{c} \sum Q_i Q_j Q_\ell$ belongs to $I^7$.  Again, this cannot be improved, as can be seen by taking $Q_1, Q_2, Q_3$ to be ``generic" quaternion algebras, and $Q_4$ to be split.  
 \end{borel*}
 
 \smallskip
 However, the Killing forms of the groups $G$ constructed from \eqref{const.map} are mainly of interest in case the Rost invariant $r(G)$ is zero. That case is treated by the following theorem provided by Detlev Hoffmann.
 
\begin{thm}[Hoffmann] \label{DWH}
Let $G$ be a group of type $E_8$ constructed via \eqref{const.map} from quaternion algebras $Q_1, Q_2, Q_3, Q_4$ and $c \in \kx$.  If $r(G) = 0$, then
\[
\kappa(G) = 2\pform{c} Q_1 Q_2 Q_4 \quad \in I^8.
\]
\end{thm}

\begin{proof}
The hypothesis implies that $\pform{c} \sum Q_i$ belongs to $I^4$, hence that $\pform{c} (Q'_1 - Q'_2)$ and $\pform{c} (Q'_3 - Q'_4)$ are congruent mod $I^4$, so \cite[Cor.]{Hoff:izh}\footnote{The statement of this result includes the hypothesis that the 12-dimensional form $\pform{c} (Q'_1 - Q'_2)$ is anisotropic, but it is unnecessary.} gives that
\[
\parbox{4in}{$\pform{c} (Q_1 - Q_2) = \qform{m} \pform{c} (Q_3 - Q_4)$ for some $m \in \kx$.}
\]
This implies that
\begin{equation} \label{DWH.1}
\pform{c} (Q_1 - Q_2)^2 = \pform{c} (Q_3 - Q_4)^2
\end{equation}
in the Witt ring.  Further,
\begin{equation} \label{DWH.2}
\pform{c} Q_1 Q_2 (Q_3 - Q_4) = \qform{m} \pform{c} Q_1 Q_2 (Q_1 - Q_2) = 0
\end{equation}
and similarly
\begin{equation} \label{DWH.3}
\pform{c} Q_1 (Q_3 - Q_4)^2 = \pform{c} Q_1 (Q_1 - Q_2)^2 = 4 \pform{c} Q_1 (Q_1 - Q_2).
\end{equation}
Of course, the roles of $Q_1$ and $Q_2$ are not special, and the same identities hold for every permutation of the subscripts.

We now compute:
\begin{align*}
\pform{c} \left( 4 \sum Q_i - 2 \sum Q_i Q_j \right) &= \pform{c} (\sum Q_i^2 - 2 \sum Q_i Q_j) \\
&= \pform{c} \left( \sum_{i<j} (Q_i - Q_j)^2 - 2 \sum Q_i^2 \right).
\end{align*}
Applying \eqref{DWH.1}, we find:
\begin{align*}
\pform{c} \left( 4 \sum Q_i - 2 \sum Q_i Q_j \right) &= \pform{c} \left( 2 (Q_1 - Q_2)^2 + 2(Q_1 - Q_3)^2 + 2(Q_1 - Q_4)^2 - 2 \sum Q^2_i \right) \\
&= \pform{c} \left[ 4 Q_1^2 - Q_1 (4 Q_2 + 4Q_3 + 4Q_4) \right].
\end{align*}
We can replace the $4Q_3 + 4Q_4$ with $(Q_3 - Q_4)^2 + 2Q_3 Q_4$, and further replace that with $4Q_1 - 4Q_2 + 2Q_3 Q_4$ by \eqref{DWH.3}, i.e.,
\begin{align*}
\pform{c} \left( 4 \sum Q_i - 2 \sum Q_i Q_j \right) &= \pform{c} \left[ 4Q_1^2 - 4Q_1 Q_2 - 4Q_1^2 + 4Q_1 Q_2 - 2Q_1 Q_3 Q_4 \right] \\
&= -2\pform{c} Q_1 Q_3 Q_4.
\end{align*}

Evaluating the full Killing form, we have:
\begin{align*}
\kappa(G)  &= \pform{c} (Q_1 Q_2 Q_3 - Q_1 Q_3 Q_4 + Q_2 Q_3 Q_4 + Q_1 Q_2 Q_4) \\
&= \pform{c} \left[ Q_1 Q_2 (Q_3 - Q_4) + 2Q_1 Q_2 Q_4 - (Q_1 - Q_2) Q_3 Q_4 \right],
\end{align*}
and the claim follows by applying \eqref{DWH.2} twice.
\end{proof}

Of course, the roles of $Q_1, Q_2, Q_4$ in the theorem are not special, and one can take any three of the four quaternion algebras by \eqref{DWH.2}.

\begin{rmk}
One can view Th.~\ref{DWH} as giving a relationship amongst four symbols in $H^d(k, \Zm2)$ that sum to zero.  For three symbols, one has the Elman-Lam Linkage Thoerem \cite[X.6.22]{Lam}.
\end{rmk}

\begin{eg}
Suppose that $k$ is a field with $2^3 \cdot I^5 \ne 0$ in the Witt ring; for example, this happens if $k$ is formally real.  Then there is a 5-Pfister form $\phi$ such that $2^3 \phi \ne 0$.  We write $\phi = \pform{c} q_1 q_2$ for 2-Pfister forms $q_1, q_2$, and $c \in \kx$.  Taking
\begin{itemize}
\item construction \eqref{const.map} where $Q_1, Q_2, Q_3, Q_4$ have norms $q_1, q_2, q_1, q_2$ respectively, or
\item Tits's construction \ref{Tits.const} with a reduced Albert algebra such that $\gamma_3 = \phi_3 = \pform{c} q_1$ and $\phi_5 = \phi$,
\end{itemize}
one obtains a group $G$ of type $E_8$ over $k$ that has zero Rost invariant, has $\kappa(G) = 8\phi$ nonzero, and is anisotropic (by Prop.~\ref{index}(1)).
\end{eg}

\section{A conjecture, and its consequences} \label{conj.sec}

Consider the following statement:
\begin{equation} \label{PFC.con}
\parbox{4in}{\emph{For every odd-degree separable extension $K/k$ and every central simple $K$-algebra of degree $16$ with orthogonal involution $\As$ in $I^3K$, we have: If $e_3\As$ is zero, then there is an odd-degree separable extension $L/K$ such that $\As \ot L$ is completely decomposable.}}
\end{equation}
By Prop.~\ref{pfinv}, if $e_3\As$ is zero, then $e_3\As$ is generically Pfister.  That is, \eqref{PFC.con} is somewhat weaker than a ``yes" answer to Question \ref{4pf.ques}.  It is natural to hope that \eqref{PFC.con} holds for every field $k$ of characteristic different from 2.

We have:
\begin{thm}\label{serre.thm}
Suppose that $k$ is a field of characteristic $\ne 2, 3$ for which \eqref{PFC.con} holds.  Then
the map $G \mapsto \rkill_{E_8} - \rkill_{G}$ defines a function
\[
\fbox{\parbox{2in}{Groups of type $E_8$ over $k$ whose Rost invariant has odd order}} \ra I^8(k).
\]
\end{thm}

\begin{proof}
We first observe that there is a separable extension $K/k$ of odd degree such that $\res_{K/k}G \in H^1(K, E_8)$ is the image of some $\eta \in H^1(K, \hspin_{16})$.  This follows from the fact that $\hspin_{16}$ contains a maximal torus of $E_8$ and that the 2-Sylow subgroups of the Weyl groups have order $2^{14}$ in both cases, see e.g.\ the proof of \cite[13.7]{G:lens}.  (For $\hspin_{16}$, one checks that the 2-primary part of $2^7 \, 8!$ is $2^{14}$.)

We enlarge $K$ so that it also kills $r(G)$.  By \eqref{PFC.con}, there is an odd-degree extension $L/K$ such that the image of $\eta$ in $H^1(L, \PSO_{16})$ is the class of a tensor product $\otimes_{i=1}^4 (Q_i, \binv)$, and it follows that $\res_{L/k}G$ is in the image of \eqref{const.map}.  

By Lemma \ref{inI4}, $\kappa(G)$ is in $I^4 k$.  But over the odd-degree extension $L/k$, $\kappa(G)$ is in $I^8 L$ by Th.~\ref{DWH}.  Write $e_n$ for the invariant $I^n(*) \ra H^n(*, \Zm2)$ defined in \cite{OVV} such that for every extension $F/k$, the map $(e_n)_F$ is an additive homomorphism with kernel $I^{n+1}F$.  
Here, $e_4(\kappa(G))$ is killed by $L$ and so is zero, hence $\kappa(G)$ is in $I^5  k$.  Repeating this with $e_5, e_6$, and $e_7$ shows that $\kappa(G)$ is in $I^8 k$.
\end{proof}

\medskip

\noindent{\small{\textbf{Acknowledgments.} I thank D.W.~Hoffmann and K.~Zainoulline for providing Theorem \ref{DWH} and Prop.~\ref{kirill} respectively, and V.~Chernousov, R.~Parimala, A.~Qu\'eguiner-Mathieu, and J.-P.~Tignol  for their comments.  I gratefully acknowledge the support by NSF grant DMS-0654502.  Part of the research for this paper was done during a visit at IHES, and I thank them for providing an excellent working environment.}}

\appendix
\section{Non-hyperbolicity of orthogonal involutions\\{\protect\textrm{By Kirill Zainoulline}}} \label{kirill.sec}

The purpose of the following notes is to prove the following

\begin{prop} \label{kirill}
Let $(A, \sigma)$ be a central simple algebra with 
orthogonal involution over a field $k$ of characteristic $\ne 2$.  
If $\deg A / \ind A$ is odd, 
then the involution $\sigma$ is not hyperbolic 
over the function field of the Severi-Brauer variety of $A$.  
\end{prop}

This result can also be deduced from the divisibility of the Witt index of $(A,\sigma)$ proved recently by N.~Karpenko (see \cite[Th.~3.3]{Karp:iso}).
Our arguments use the notion of the $J$-invariant instead.

\begin{proof}[Proof of Prop.~\ref{kirill}]
The case where $A$ has index 1 is clear and 
the index 2 case is \cite[Prop.~3.3]{PSS:herm},
so we may assume that $A$ has index at least 4 and hence 
degree is divisible by 4.  
Further, we may assume that $(A, \sigma)$ is in $I^3$ as in Example \ref{In.egs}\eqref{In.3}, 
otherwise the conclusion is obvious. 

Consider the groups 
$G=\hspin(A,\sigma)$ and $G'=\PGL(A)$.
Let $X$ and $X'$ be the respective varieties of Borel subgroups.
\begin{equation}\label{kirill:mainas}
\text{Assume that $\sigma$ is hyperbolic over the function
field $k(SB(A))$.}
\end{equation} 
Then the group $G$ is split over $k(SB(A))$ and, hence, over $k(X')$.
Since the group $G$ is split over $k(X)$, the algebra $A$ and the group $G'$ 
are split over $k(X)$.

By the main result of the paper \cite{Jinv} (Theorems 4.9 and 5.1)
to any simple linear algebraic
group of inner type over $k$ and its torsion prime $p$
one may associate an indecomposable Chow motive $\mathcal{R}_p$ such that
over the algebraic closure $\bar k$ of $k$ the generating function of $\mathcal{R}_p$
is given by the product of $r$ cyclotomic polynomials
\[
\prod_{i=1}^r 
\frac{1-t^{d_i2^{j_i}}}{1-t^{d_i}}
\text{ , where }0\le j_i\le k_i\text{ and }d_i>0
\]
and the explicit values of the parameters $d_i$ and bounds $k_i$ 
are provided in \cite[Table~6.3]{Jinv}.
The $r$-tuple of integers $(j_1,j_2,\ldots,j_r)$ is called the $J$-invariant.

Let $\mathcal{R}_2(G)$ and $\mathcal{R}_2(G')$ 
be the respective motives for the groups $G$ and 
$G'$ and for $p=2$. By \cite[Prop.~5.3]{Jinv} applied
to $G$ over $k(X')$ and $G'$ over $k(X)$
we obtain the following motivic reformulation 
of the assumption \eqref{kirill:mainas}:
\begin{equation}\label{kirill:moteq}
\mathcal{R}_2(G)\simeq \mathcal{R}_2(G').
\end{equation}

Since the group $G'$ is a twisted form of the group $\PGL_{\deg A}$, 
by the first line of \cite[Table~6.3]{Jinv}
the $J$-invariant of $G'$ has only one entry ($r=1$), 
and the parameter $d_1$ is $1$.  
Then by the proof of \cite[Lemma 7.3]{Jinv},
we obtain that the $J$-invariant is the 
list consisting of the single element $s$, 
where $2^s$ is the index of $A$.  
Hence the generating function of $\mathcal{R}_2(G')$ is 
$(1-t^{2^s})/(1-t)$.

Similarly, since the group $G$ is a twisted form of the group 
$\hspin_{\deg A}$,
by \cite[Table~6.3]{Jinv} the $J$-invariant of $G$
has $\tfrac{1}{4}\deg A$ entries with $d_i=2i-1$ and the following
inequality holds
$$
j_1\le k_1=v_2(\tfrac{1}{2}\deg A)=v_2(2^{s-1}\cdot\tfrac{\deg A}{\ind A})=
s-1<s,
$$ 
where $v_2$ is the 2-adic valuation.  (Here we essentially 
use that $\tfrac{\deg A}{\ind A}$
is odd.)

The isomorphism \eqref{kirill:moteq} implies the equality
of the respective generating functions, namely
$$
\frac{1-t^{2^s}}{1-t}=\prod_{i=1}^{s/2} 
\frac{1-t^{(2i-1)2^{j_i}}}{1-t^{2i-1}}\text{ , where }j_1 <s.
$$
We claim that it never holds.
Indeed, comparing the coefficients at $t^2$ and $t^3$ of the polynomials 
at the left and the right hand side, we conclude that $j_1\ge 2$ and $j_2=0$.
Then comparing them consequently at powers $t^{2i-2}$ and $t^{2i-1}$, $i\ge 3$ 
we conclude that $2^{j_1}\ge 2i-1$ and $j_i=0$.
Therefore, $j_3=\cdots=j_{s/2}=0$ and $j_1$
must coincide with $s$, which is not the case, since $j_1<s$.

Hence, the assumption
\eqref{kirill:mainas} fails and the lemma is proven.
\end{proof}

\input deg16-bbl


\end{document}

%% file: skip-thm.tex
\usepackage{amsthm}

\newtheorem{thm}[equation]{Theorem}
\newtheorem{lem}[equation]{Lemma}
\newtheorem{cor}[equation]{Corollary}
\newtheorem{prop}[equation]{Proposition}

\newtheorem*{thm*}{Theorem}
\newtheorem*{prop*}{Proposition}
\newtheorem*{cor*}{Corollary}
\newtheorem*{lem*}{Lemma}
\newtheorem*{MT*}{Main Theorem}

\theoremstyle{definition} %
\newtheorem{defn}[equation]{Definition}
\newtheorem*{defn*}{Definition}
\newtheorem{ques}[equation]{Question}

\newtheorem{eg}[equation]{Example}

\theoremstyle{remark} %
\newtheorem{rmk}[equation]{Remark}

\newtheorem*{rmk*}{Remark}
\newtheorem*{rmks*}{Remarks}

\newtheoremstyle{exercise}
  {3pt}
  {3pt}
  {\small}
  {}
  {\sc\small}
  {.}
  {.5em}
   {}     
  {}

\theoremstyle{exercise}

\makeatletter
{\renewcommand{\theequation}{#1}}%
{\renewcommand{\theequation}{\arabic{equation}}\addtocounter{equation}{-1}\global\@ignoretrue}
\makeatother

\makeatletter
{\renewcommand{\theequation}{#1}\begin{eqnarray}}%
{\end{eqnarray}\renewcommand{\theequation}{\arabic{equation}}\addtocounter{equation}{-1}\global\@ignoretrue}
\makeatother

\makeatletter
\newenvironment{borel}[1]%
{\smallskip \refstepcounter{equation}\noindent{\textbf{\theequation.} }{{\textbf{#1.}}}}%
{\smallskip \global\@ignoretrue}
\makeatother

\makeatletter
{\smallskip \refstepcounter{equation}{\sc \theequation}{\sc (#1).}}%
{\smallskip \global\@ignoretrue}
\makeatother

\makeatletter
{\smallskip \refstepcounter{equation}\noindent{\sc \theequation.}{\sl{ #1.}}}%
{\smallskip \global\@ignoretrue}
\makeatother

\makeatletter
\newenvironment{borel*}%
{\smallskip \refstepcounter{equation}\noindent{\textbf{\theequation.}}}%
{\global\@ignoretrue}
\makeatother

\newcommand{\flist}[1]{\hangindent\leftmargini\textup{(1)}\hskip\labelsep {#1}%
\begin{enumerate}%
\setcounter{enumi}{1}%
}

%% file: qformsl-def.tex
\newcommand{\ot}{\otimes}

\newcommand{\darkrad}{0.115}
\newcommand{\lrad}{0.25}

\newcommand{\eand}{\quad\text{and}\quad}

\newcommand{\R}{{\mathbb{R}}}        
\newcommand{\Z}{{\mathbb{Z}}}        
\renewcommand{\H}{{\mathcal{H}}}  
\newcommand{\hyp}{\mathcal{H}}

\newcommand{\QZt}{\mathbb{Q}/\mathbb{Z}(2)}

\newcommand{\Zm}[1]{\Z/{#1}\Z}

\newcommand{\e}{\varepsilon}

\newcommand{\la}{\lambda}
\newcommand{\D}{\Delta}

\newcommand{\G}{{\Gamma}}       

\newcommand{\oddots}{{\mathinner{\mkern1mu\raise1pt\vbox{\kern7pt\hbox{.}}\mkern2mu\raise4pt\hbox{.}\mkern2mu\raise7pt\hbox{.}\mkern1mu}}}

\newcommand{\s}{\sigma}

\newcommand{\As}{{(A,\sigma)}}

\newcommand{\Bt}{{(B,\tau)}}

\newcommand{\supast}[1]{{{#1}^\times}}
\newcommand{\xsq}[1]{{{#1}^{\times2}}}
\newcommand{\sq}[1]{{\supast{#1}/\xsq{#1}}}

\newcommand{\kx}{k^\times}
\newcommand{\ksq}{\sq{k}}

\newcommand{\qform}[1]{{\left\langle{#1}\right\rangle}}                   
\newcommand{\pform}[1]{{\langle\!\langle{#1}\rangle\!\rangle}} 

\newcommand{\basemu}{\boldsymbol{\mu}}
\newcommand{\mmu}[1]{\basemu_{#1}}     

\DeclareMathOperator{\Spin}{Spin}           
\newcommand{\Sp}{\mathrm{Sp}}
\DeclareMathOperator{\SL}{SL}

\newcommand{\SO}{\mathrm{SO}}

\newcommand{\Ga}{\mathbb{G}_a}
\newcommand{\Gm}{\mathbb{G}_m}

\DeclareMathOperator{\tr}{tr}

\DeclareMathOperator{\im}{im}

\DeclareMathOperator{\Int}{Int}
\DeclareMathOperator{\chr}{char}

\DeclareMathOperator{\Id}{Id}

\DeclareMathOperator{\res}{res}

\DeclareMathOperator{\aut}{Aut}

\newcommand{\iso}{\xrightarrow{\sim}}
\newcommand{\injects}{\hookrightarrow}

\newcommand{\ra}{\rightarrow}

\newcommand{\tbtmat}[4]{\left( \begin{array}{cc} #1&#2 \\ #3&#4 \end{array} \right) }

\newcommand{\stcolvec}[2]{\left( \begin{smallmatrix} #1 \\ #2 \end{smallmatrix} \right) }
\newcommand{\stbtmat}[4]{\left( \begin{smallmatrix} #1&#2 \\ #3&#4 \end{smallmatrix} \right) }

%% file: deg16-bbl.tex
\providecommand{\bysame}{\leavevmode\hbox to3em{\hrulefill}\thinspace}
\providecommand{\MR}{\relax\ifhmode\unskip\space\fi MR }
\providecommand{\MRhref}[2]{%
  \href{http://www.ams.org/mathscinet-getitem?mr=#1}{#2}
}
\providecommand{\href}[2]{#2}